\documentclass[12pt]{amsart}

\makeatletter
\def\@tocline#1#2#3#4#5#6#7{\relax
\ifnum #1>\c@tocdepth 
\else
\par \addpenalty\@secpenalty\addvspace{#2}%
\begingroup \hyphenpenalty\@M
\@ifempty{#4}{%
\@tempdima\csname r@tocindent\number#1\endcsname\relax
}{%
\@tempdima#4\relax
}%
\parindent\z@ \leftskip#3\relax \advance\leftskip\@tempdima\relax
\rightskip\@pnumwidth plus4em \parfillskip-\@pnumwidth
#5\leavevmode\hskip-\@tempdima
\ifcase #1
\or\or \hskip 1em \or \hskip 2em \else \hskip 3em \fi%
#6\nobreak\relax
\dotfill\hbox to\@pnumwidth{\@tocpagenum{#7}}\par
\nobreak
\endgroup
\fi}
\makeatother

\usepackage{amssymb}
\usepackage{amsthm}
\usepackage{amsmath}
\usepackage[mathcal]{euscript}
\usepackage[margin=1in]{geometry}
\usepackage{hyperref}
\usepackage{tikz}
\usetikzlibrary{matrix}	

\newtheorem{theorem}{Theorem} 
\newtheorem{prop}{Proposition}
\newtheorem{definition}{Definition}
\newtheorem{lemma}{Lemma}
\newtheorem{corollary}{Corollary}
\newtheorem{example}{Example}
\newtheorem*{theorem*}{Theorem}

\numberwithin{equation}{section}
\numberwithin{theorem}{section}
\numberwithin{lemma}{section}
\numberwithin{corollary}{section}
\numberwithin{example}{section}

\title{The Asymptotics of Representations for Cyclic Opers}
\author{Jorge Acosta }
\date{\today}

\begin{document}

\begin{abstract}
Given a Riemann surface $X = (\Sigma, J)$ we find an expression for the dominant term for the asymptotics of the holonomy of opers over that Riemann surface corresponding to rays in the Hitchin base of the form $(0,0,\cdots,t\omega_n)$. Moreover, we find an associated equivariant map from the universal cover $(\tilde{\Sigma},\tilde{J})$ to the symmetric space SL$_n(\mathbb{C}) / \mbox{SU}(n)$ and show that limits of these maps tend to a sub-building in the asymptotic cone. That sub-building is explicitly constructed from the local data of $\omega_n$.
\end{abstract}
\maketitle
\tableofcontents
\section{Introduction}
\label{intro}

Let $\Sigma_g$ be a compact surface of genus $g \geq 2$ and $X = (\Sigma_g, J)$ be a Riemann surface structure on $\Sigma_g$. The SL$_n(\mathbb{C})$ character variety is defined to be $$\mathcal{M}_B^{(n)} = \mbox{Hom}(\pi_1(X),\mbox{SL}_n(\mathbb{C}))//\mbox{SL}_n(\mathbb{C}).$$ We study the space of opers, a subspace of $\mathcal{M}_B^{(n)}$ defined in terms of ordinary differential equations on X. The Hitchin base, a vector space of differentials, parametrizes the space $Op(n)$, of SL$_n-$opers on X and the map that sends an oper to its holonomy H: $Op(n) \rightarrow \mathcal{M}_B^{(n)}$ is a proper embedding (see \cite{Wentworth} for details). The goal of this paper is to understand the asymptotics of this map. That is, given a ray v$: \mathbb{R} \rightarrow Op(n)$, what can we say about the asymptotics of corresponding curve H(v(t)) in $\mathcal{M}_B^{(n)}$? We will focus on the case $$\mbox{v}(t) = (0,... , t\omega_n)$$ where the holomorphic $n$ differential $\omega_n \in \mbox{H}^0(X,\mathcal{K}^n)$ on $X$ has simple zeros. The holonomy of SL$_n-$opers (described more completely in section \ref{bg}) is the holonomy of a differential equation on the universal cover $\tilde{X}$. By understanding the asymptotics of the solutions of the differential equation, we can use this information to compute the asymptotics of the resulting holonomy. There are two statements one can make about the asymptotics of the holonomy an analytic statement and a geometric statement.  We relate the asymptotics of the holonomy to the local geometry of the defining holomorphic differential and to the Stokes matrices (section \ref{bg}) of the differential equation. More precisely, given a holomorphic $n$ differential $\phi$ and a curve $\gamma$ on $\tilde{X}$, if we define $E_t(\gamma,\phi) = e^{t^{\frac{1}{n}} \int_{\gamma} A(\phi)}$, where $A(\phi)$ = diag$(\phi^{\frac{1}{n}}, \cdots, \lambda^{n-1}\phi^{\frac{1}{n}})$ we can compute the asymptotics of the holonomy in terms of the matrices $E_t(\gamma,\phi)$. That is, we have the following analytic theorem. 

\begin{theorem}

\label{thm1}
Let $\rho_t$ be the holonomy associated to $(0,... , t\omega_n)$ by the map above and $[\gamma] \in \pi_1(X)$. Then there exists a decomposition of $\gamma = \gamma_1 +\cdots+ \gamma_N$ and a collections of Stokes matrices $A_i$ so that $$\lim_{t \to \infty}\rho_t([\gamma])( \prod_{i=1}^{N} E_t(\gamma_{i} ,\omega_{n} )A_i)^{-1} = I_n$$
\end{theorem}

For each $t \in \mathbb{R}$, the oper defined by $t\omega_n$ gives both a holonomy $\rho_t$ and a differential equation $D_t(y) = 0$. The holonomy $\rho_t$ can be represented geometrically by a $\rho_t$ equivariant map to the associated symmetric space SL$_n(\mathbb{C})/\mbox{SU}(n)$. We then can understand the asymptotics of $\rho_t$ geometrically by the limits of the equivariant maps. We first show the following statement. 

\begin{prop}
In terms of the solutions to $D_t(y) =0$ there exists an explicit map $$Ep_t: \tilde{X} \setminus Z_{\omega_n} \rightarrow SL_n(\mathbb{C})/SU(n)$$ which is $\rho_t$ equivariant.
\end{prop}

Then the asymptotics of the solutions provide for a limiting map to the asymptotic cone of $\mbox{SL}_n(\mathbb{C})/\mbox{SU}(n)$. More precisely we show that

\begin{theorem}
\label{thm2}
The family of maps $Ep_t$ induce a map $$Ep: \tilde{X} \rightarrow Cone( SL_n(\mathbb{C}) / SU(n), \{ Id \}, \{ k^{\frac{1}{n}} \} ) $$ such that for each $z \notin Z_{\omega_n}$ there exist an open set $U$ containing z so that $Ep(U) \subset A$ an apartment in $Cone(SL_n(\mathbb{C}) / SU(n))$.
\end{theorem}
The notation used in the theorem above will be explained in section \ref{bg}. These results work towards a generalization of Dumas \cite{Dumas} in the case of cyclic Opers. Similar results in different context can be seen in  \cite{collier2014asymptotics},\cite{dumas2015polynomial} \cite{katzarkov2015constructing}, \cite{KAPS}, \cite{loftin2007flat},\cite{wolf1991high}.
\subsection*{Outline of methods}
We prove Theorem \ref{thm1} using two lemmas. One lemma computes the local asymptotics of the solution to the differential equation at a point that is not a zero of $\omega_n$. The second lemma computes the local asymptotics of the differential equation in a neighborhood of a zero of $\omega_n$. Then given a curve $\gamma$ on $\tilde{X}$, we glue together the local asymptotics to obtain asymptotics of the holonomy $\rho_t(\gamma)$. 

To prove Theorem \ref{thm2} we construct the $\rho_t$-equivariant map $Ep_t$ as the product of the matrix Wronskian  for the differential equation and a correction matrix that depends only on $\omega_n$. Finally applying the two lemmas to the matrix Wronskian, we see that near a point $z \notin Z_{\omega_n}$, the zero set of $\omega_n$, the image of the map $Ep_t$ approaches a flat SL$_n(\mathbb{C})/\mbox{SU}(n)$ and at a zero of the $n$ differential $\omega_n$ the image of $Ep_t$ approaches a collection of flats depending on Stokes data. 

\subsection*{Outline}
Section \ref{bg} will give background on SL$_n$-opers, asymptotic expansions, the Stokes phenomenon and the asymptotic cone of a metric space. In section \ref{lems} we will prove the two lemmas on the local asymptotics of corresponding differential equations, one giving the asymptotics of $D_t(y) = 0$ away from a zero of $\omega_n$ and the second giving the asymptotics of $D_t(y) = 0$ at a zero of $\omega_n$. Section \ref{PROOF} will contain the proof of the Theorem 1.1. Finally in section \ref{building} we will describe how to define the equivariant map, $Ep_t$, using the solutions to the ODE, and then verify Theorem  \ref{thm2}.

\subsection*{Acknowledgements}
The author would like to thank his advisor Michael Wolf for many helpful discussions and encouragement. The author also acknowledges support from U.S. National Science Foundation grants DMS 1107452, 1107263, 1107367 "RNMS: Geometric Structures and Representation Varieties" (the GEAR Network).
\section{Background}

\label{bg}
\subsection{Opers}
This section contains a brief introduction to SL$_n(\mathbb{C})$ opers on a Riemann surface. Here we define opers, explain the relationship to ODE and describe how to compute the corresponding holonomy. Let X be a Riemann surface structure denote the holomorphic cotangent bundle as $\mathcal{K}$ and the space of holomorphic $n$-th differentials by H$^0(X,\mathcal{K}^n)$. The Hitchin base of X is defined to be $$\mathcal{H}_n = \bigoplus_{i=2}^{n} \mbox{H}^0(X,\mathcal{K}^i).$$
A SL$_n$-oper on a Riemann surface X is defined as follows. 

\begin{definition}
An $SL_n$-oper $\mathcal{D} = \mathcal{D}(V,\nabla, \mathcal{F})$ is a holomorphic vector bundle $V$, holomorphic connection $\nabla$ inducing the trivial connection on det V, and a holomorphic filtration $\mathcal{F}$ $$0 = V_0 \subset V_1 ...\subset V_n = V$$ such that \begin{itemize}
\item $\nabla V_i \subset V_{i+1} \otimes \mathcal{K}$
\item The induced map $\nabla: {V_i}/{V_{i-1}} \rightarrow \frac{V_{i+1}}{V_i} \otimes \mathcal{K}$ is an isomorphism. \end{itemize}
\end{definition}

This is equivalent to a local system $\mathbf{V}$ (see \cite{Wentworth} for details) that is \textit{realized in} $\mathcal{K}^\frac{1-n}{2}$. That is, we have a short exact sequence of sheaves \begin{equation*} 0 \rightarrow \mathbf{V} \rightarrow \mathcal{K}^\frac{1-n}{2} \xrightarrow{D} \mathcal{K}^\frac{1+n}{2} \rightarrow 0, \end{equation*} where $D$ is a differential operator between the line bundles $\mathcal{K}^\frac{1-n}{2}$ and $\mathcal{K}^\frac{1+n}{2}$. Lifting to the universal cover $\tilde{X}$, the operator $D$ defines a differential equation D(y) = 0 on $\tilde{X}$. There is a one-to-one correspondence between such differential equations and the Hitchin base $\mathcal{H}_n$ (see \cite{Walgebras}). That is, given $(q_2,q_3,\cdots,q_n) \in \mathcal{H}_n$ we obtain an ODE on $\tilde{X}$ 

\begin{equation} \label{eq0} D(y) = y^{(n)} + Q_2y^{(n-2)}+\cdots + Q_n y = 0.
\end{equation} 

Let $Op(n)$ denote the space of gauge equivalence classes of SL$_n$-opers on $X$. The above correspondence shows that $ \mathcal{H}_n$ parametrizes $Op(n)$. In \cite{baraglia2010cyclic} Baraglia calls a Higgs bundle \textit{cyclic} if, in the projection to the Hitchin base, all but the highest differential vanishes. Following this we will call a Oper \textit{cyclic} if, in the above parametrization, all but the highest differential vanishes.

An example of this parametrization can be seen for $n = 2$ where $Op(2)$ corresponds to the space of $\mathbb{CP}^1$ structures on X \cite{Dumas}. We now provide some details in this case.
\begin{definition} A $\mathbb{CP}^1$ structure on X is defined, up to Mobius transformations, as a pair $(f, \rho)$ where $$\rho: \pi_1(X) \rightarrow \mbox{PSL}_2(\mathbb{C})$$ and $$f: \tilde{X} \rightarrow \mathbb{CP}^1$$ a locally injective equivariant holomorphic function.
\end{definition}  From this data we can consider the differential equation $Dy = y'' + S(f)y = 0$ on $\tilde{X}$, where $$S(f) = \left(\frac{f''}{f'}\right)' - \frac{1}{2}\left(\frac{f''}{f'}\right)^2$$  is the Schwarzian derivative. Now, in the more contemporary language of maps between bundles, it is easy to check that $Dy$ defines a operator between $\mathcal{K}^{-\frac{1}{2}}$ and $\mathcal{K}^{\frac{3}{2}}$ and that $\rho$ is the monodromy of the differential equation. Conversely $(f,\rho)$ is determined by $S(f)$: up to Mobius transformations $f$ is the ratio of two linearly independent solutions of $Dy=0$. This shows that the moduli space of projective structures on X is parametrized by the space of Schwarzian derivatives; this space is known to be an affine space modelled on the space of quadratic differentials, which for $n=2$ is $\mathcal{H}_2$.

\subsubsection*{Associated holonomy}
Fix $\vec{q} = (q_2,\cdots,q_n) \in \mathcal{H}_n$ and let $\mathcal{D} = \mathcal{D}(\vec{q})$ be the corresponding oper.  To each oper, we have a flat connection $\nabla$ from which there is a holonomy representation $$\rho: \pi_1(X) \rightarrow \mbox{SL}_n(\mathbb{C}).$$ In this section we describe $\rho$ in terms of the solutions to equation \eqref{eq0}. The equation \eqref{eq0} is equivalent to a matrix equation

 \begin{equation}
 \label{meq}
 Y' = \left(\begin{array}{cccc} 0 & 1 & & \\ & 0 & 1 & \\ & & \ddots & 1 \\ -Q_n & -Q_{n-1} & & 0 \end{array} \right)Y = A(z)Y. \end{equation}
 
More precisely, $Y(z)$ is a matrix solution to \eqref{meq} with det$(Y(z)) \neq 0$ if and only if $Y(z) = W(y_1, \cdots , y_n)^T$ where $W(y_1, \cdots , y_n)$ is the matrix Wronskian of $y_1(z),..., y_n(z)$ a basis of solutions to \eqref{eq0}. 
\begin{definition}
Let $n$ linearly independent functions, $f_1(z),\cdots,f_n(z)$ be given, the matrix Wronskian is defined to be
$$W(f_1(z),\cdots,f_n(z)) = \left(\begin{array}{cccc} f_1(z) & f_1'(z) & \cdots & f_1^{(n-1)}(z) \\ f_2(z) & f_2'(z) & \cdots & f_2^{(n-1)}(z) \\ . & . & . & \\ f_n(z) & f_n'(z) & \cdots & f_n^{(n-1)}(z) \end{array} \right).$$
\end{definition}
We note here two useful facts about the Wronskian. Given functions $f_1,f_2, \cdots,f_n,g,h$, since $(fg)^j$  can be written as a linear combination of $f^i$ with coefficients in terms of $g$ and $ (f\circ h)^j$ can be written as linear combination of $(f^i \circ h)$ with coefficients in terms of h, there exists unique matrices $M_1(g), M_2(h)$ such that 
\begin{itemize}
\item $W(gf_1,gf_2,\cdots,gf_n) = W(f_1,\cdots,f_n)M_1(g)$
\item $W(f_1 \circ h, \cdots, f_n \circ h) = W(f_1,\cdots,f_n)(h) M_2(h')$
\end{itemize}
\emph{Example} In the case $n = 3$ we would have
$$M_1(g) = \left( \begin{array}{ccc} g & g' & g'' \\ 0 & g & 2g' \\ 0 & 0 & g \end{array} \right), M_2(h') = \left( \begin{array}{ccc} 1 & 0 & 0 \\ 0 &h' & h'' \\ 0 & 0 & h'^2 \end{array} \right).$$
We prove here a fact that will be useful in section \ref{building}. 
\begin{prop} \label{induct}
The matrix $M = M_2(h')M_1(g) = (M_{ij})$ satisfies the following properties 
\begin{enumerate}
    \item  $M_{00}(z) = g$
    \item  for $i >j \quad M_{ij} = 0$

    \item for $ i \leq j \quad M_{ij} = M_{i,j-1}'(z)+  M_{i-1,j-1}(z)h'	.$
\end{enumerate}
\end{prop}
\begin{proof}
The matrix $M$ is the unique matrix which satisfies \begin{equation} \label{eqWrons}
W(g(f_1 \circ h), \cdots,g(f_n \circ h)) = W(f_1,\cdots,f_n)(h)M.\end{equation}
Let $\hat{M}$ be the matrix defined by properties (1) - (3). We will show that $\hat{M}$ satisfies \eqref{eqWrons} and hence $M$ satisfies properties (1) - (3). We need to show that 
\begin{align*}
W(g(f_1 \circ h), \cdots,g(f_n \circ h))_{ij} &= (W(f_1,\cdots,f_n)(h)\hat{M})_{ij} \\
 (g(f_i \circ h))^{(j)} &= \sum_k (f_i^{k} \circ h)\hat{M}_{kj}.
\end{align*}
We will prove this by induction on $j$. For the base case $j = 0$ we have $$g(f_i \circ h) =  \sum_k (f_i^{k} \circ h)\hat{M}_{kj}$$ since $\hat{M}_{00} = g$ and $\hat{M}_{i0} = 0$ for $i \neq 0$.
Suppose that  $(g(f_i \circ h))^{(j)} = \sum_k (f_i^{k} \circ h)\hat{M}_{kj}$ for some $j < n$
Then
\begin{align*}
(g(f_i \circ h))^{(j+1)} &= ((g(f_i \circ h))^{(j)})' \\
&= (\sum_k (f_i^{k} \circ h)\hat{M}_{kj})' \\
&= \sum_k (f_i^{k+1} \circ h)h'\hat{M}_{kj} +\sum_k (f_i^{k} \circ h)\hat{M}'_{kj} \\ 
&= \sum_k (f_i^{k+1} \circ h)\left[ \hat{M}_{k+1,j+1} - \hat{M}'_{k+1.j} \right]+ \sum_k (f_i^{k} \circ h)\hat{M}'_{kj} \\
&= \sum_k (f_i^{k+1} \circ h)\hat{M}_{k+1,j+1}
\end{align*}
which is what we wanted to show. By induction we have $$W(g(f_1 \circ h), \cdots,g(f_n \circ h)) = W(f_1,\cdots,f_n)(h)\hat{M}.$$
\end{proof}
To describe the holonomy $\rho$ in terms of solutions to \eqref{eq0}, let a homotopy class $[\gamma] \in \pi_1(X)$ be given, we may associate to it a deck transformation $$\gamma : \tilde{X} \rightarrow \tilde{X}.$$ One can show that given a basis of solutions $y_1(z), \cdots, y_n(z)$to \eqref{eq0}, $$y_1(\gamma(z))\gamma'^{\frac{1-n}{2}},\cdots, y_n(\gamma(z))\gamma'^{\frac{1-n}{2}}$$ is also a basis of solutions to \eqref{eq0}. We have that both  $W(y_1(\gamma(z))\gamma'^{\frac{1-n}{2}},\cdots, y_n(\gamma(z))\gamma'^{\frac{1-n}{2}})^T$ and $W(y_1,\cdots,y_n)^T$ are solutions to \eqref{meq}. This implies that there is a change of basis matrix, $\rho(\gamma)$, satisfying 

\begin{equation} \label{holonmy} \rho([\gamma]) = (W(y_1,\cdots,y_n)^T)^{-1}W(y_1(\gamma(z))\gamma'^{\frac{1-n}{2}},..., y_n(\gamma(z))\gamma'^{\frac{1-n}{2}})^T. \end{equation}

This defines the representation H$(\mathcal{D}) = \rho: \pi_1(X) \rightarrow \mbox{SL}_n(\mathbb{C})$ of the oper. This defines a map $$\mbox{H}: Op(n) \rightarrow \mathcal{M}_B^{(n)}$$ of the space of opers into the Betti moduli space. One can show that this is in fact a proper embedding (see \cite{Wentworth}). In the case $n = 2$ boundary points of $\mathcal{M}_B^{(n)}$ are represented by $\pi_1$ actions on $\mathbb{R}$-trees this is known as the Morgan-Shalen compactification (see \cite{MS1}). Dumas studied the asymptotics of the map H in \cite{Dumas}. That is Dumas showed
\begin{theorem}[Dumas \cite{Dumas}] \label{Dumas} If $ \phi_n \in \mathcal{H}_2$ is a sequence of holomorphic quadratic differentials with projective limit $\phi \in \mathcal{H}_2$. then any accumulation point of $H(\phi_n)$ in the Morgan-Shalen boundary is represented by an $\mathbb{R}$-tree T that admits an equivariant, surjective straight map $T_{\phi} \rightarrow T$. 
\end{theorem} To prove Theorem \ref{Dumas}, Dumas, using a construction from \cite{Epstein}, defined an equivariant map Ep$_n: \tilde{X} / Z_{\phi_n} \rightarrow \mathbb{H}^3$ and showed that the limit of these maps to Cone$(\mathbb{H}^3)$ factors through the $\mathbb{R}$-tree $T_{\phi}$ associated to $\phi$.

\subsection{Asymptotic Expansion}

An asymptotic expansion describes the asymptotic behavior of a function. Here we define what it means for a function $f(z)$ to have an asymptotic expansion and list some useful properties of asymptotic power series. 

\begin{definition}
A function $f(z)$ is said to have asymptotic expansion $\sum_{r=0}^{\infty} a_r(z-a)^r$ in a set S at $a \in$ S if for each $m \in \mathbb{N}$ $$\lim_{z \to a} (z-a)^{-m}\left[ f(z) - \sum_{r=0}^{m} a_r (z-a)^r \right] = 0$$ We denote this relation by $f(z) \sim \Sigma_{r=0}^{\infty} a_r (z-a)^r$.
\end{definition}

If $f(z)$ is given by a convergent series $$f(z) = \sum_{i=0}^{\infty} a_r(z-a)^r$$ then it is easy to see that $f(z) \sim \sum_{i=0}^{\infty} a_r(z-a)^r$. If the asymptotic series does not converge there still exists function with that asymptotic series. In fact, given any series $\sum_{i=0}^{\infty} a_r (z-a)^r$ there exists a function $f(z)$ so that $f(z) \sim \sum_{i=0}^{\infty} a_r (z-a)^r$. For a given function the asymptotic expansion is uniquely determined by $$a_m = \lim_{z \to a} (z-a)^{-m}\left[ f(z) - \sum_{r=0}^{m-1} a_r (z-a)^r \right].$$ The converse is not true, an asymptotic expansion does not uniquely determine a function. For example both $f(z) = 0$, and $g(z) = e^{-\frac{1}{z}}$ have asymptotic expansion at $0$ given by $\sum_{r=0}^{\infty} 0 z^r.$

Asymptotic expansion satisfy some useful algebraic properties. Let $f(z), g(z)$ have asymptotic expansions $\sum_{r=0}^{\infty} a_r(z-a)^r$, $\sum_{r=0}^{\infty} b_r(z-a)^r$ respectfully then we have the following $$\alpha f(z) + \beta g(z) \sim \sum_{r=0}^{\infty} (\alpha a_r + \beta b_r )(z-a)^r$$ $$ f(z)g(z) \sim \sum_{r=0}^{\infty} c_r (z-a)^r $$ where $c_r = \sum_{j=0}^{r} a_j b_{r-j}$.

\subsection{Stokes Phenomenon}
Equation \eqref{holonmy} shows that the asymptotics of the representation are determined by the asymptotics of the solutions to \eqref{eq0}. This is where the Stokes phenomenon comes into play. Solutions to certain differential equations can by approximated by multivalued exponential functions that are defined explicitly in terms of the coefficients of the differential equation. However, because the solutions to the differential equations are single-valued, the approximations are valid only in certain sectors. The Stokes phenomenon describe these sectors as being bounded by \emph{Stokes lines} and relate the asymptotics between them with \emph{Stokes matrices}. (see \cite{Wasow} for details)

\begin{theorem}[Birkhoff \cite{birkhoff1909singular}]
\label{Wasow}
Consider a linear differential equation of the form 

\begin{equation} \label{Stokes1} z^{-q}Y'(z) = A(z)Y(z)
\end{equation}
where $Y(z), A(z)$ are $n$ x $n$ matrices with det$(Y(z)) \neq 0$ and $A(z)$ has asymptotic expansion $$A(z) \sim \sum_{r = 0}^{\infty} A_r z^{-r}, \quad z \to \infty, \quad z \in S$$ with $A_0$ having distinct eigenvalues $\lambda_1, \cdots, \lambda_n.$ Then for any open sector $S$ based at the origin and angle $\theta < \frac{\pi}{q+1}$ there exists a solution $Y(z)$ to \eqref{Stokes1} so that
\begin{equation} \label{asymp} Y(z) = \hat{Y}(z)z^{D}e^{Q(z)} \end{equation}
where $D$ is a constant diagonal matrix, $Q(z)$ is a diagonal polynomial with leading term $\frac{x^{q+1}\lambda_i}{q+1}$ and the matrix $\hat{Y}(z)$ has asymptotic expansion in $S$ $$\hat{Y}(z) \sim \sum_{r=0}^{\infty} \hat{Y}_r z^r, \quad z \to \infty, \quad \det \hat{Y}_0 \neq 0.$$
\end{theorem}
This expression for $Y(z)$ shows that the asymptotics of $Y(z)$ are determined by the term $e^{Q(z)}$  and so its leading term $\frac{x^{q+1}}{q+1}\lambda_i$ define Stokes lines for \eqref{Stokes1}. 
\begin{definition} The Stokes lines for \eqref{Stokes1} are given by $$l_{ij} = \{ Re(x^{q+1}(\lambda_i - \lambda_j)) = 0 \}.$$ for each $i \neq j$ \end{definition}
The collection $\{ l_{ij} \}$ of Stokes lines cut the plane into a finite number of sectors $\{ S_\alpha \}$. For each such sector $S_\alpha$, there exists a solution $Y_\alpha$ satisfying \eqref{asymp} in $S_\alpha$. Two solutions from adjacent sectors can be related by Stokes matrices. Let I$_n$ be denote the $n x n$ identity matrix and $E_{ij}$ denote the matrix with 1 in the $ij$-entry and 0 elsewhere. 
\begin{definition} Let sectors $S_\alpha$ and $S_\beta$ be adjacent and separated by a separation ray $l_{ij}$. Then given $Y_\alpha$ there exists a matrix of the form $$A_{ij} = \mbox{I}_{n} + aE_{ij}$$ for some $a \in \mathbb{C}$ so that $Y_\beta = Y_\alpha A_{ij}$ satisfies \eqref{asymp} in $S_\beta$. The matrix $A_{ij}$ is called the Stokes matrix associated to the Stokes line $l_{ij}$. \end{definition}

\subsection{Asymptotic Cones}
In section \ref{building} we consider a family of equivariant maps from $\tilde{X}$ to SL$_n(\mathbb{C})/\mbox{SU}(n)$. We interpret the limit of these maps by showing that they induce a map to the asymptotic cone Cone$(\mbox{SL}_n(\mathbb{C})/\mbox{SU}(n))$. The asymptotic cone of a metric space $X$ is an associated metric space that captures the geometry of $X$ at infinity. More precisely, let $X$ be a metric space, let $\{ x_k \}$ be a sequence of points in $X$, let $\{\lambda_k \}$ be a sequence of scalars so that $\lambda_k \to \infty$, and let $\alpha$ be a ultrafliter. 

\begin{definition}
The asymptotic cone of $(Cone(X),d) = Cone(X,\{ x_k \}, \{ \lambda_k \} )$ $X$, is a set defined by $$(Cone(X),d) = \{ (y_k) \in X^{\mathbb{N}} : \frac{d(x_k,y_k)}{\lambda_k} \mbox{ is bounded} \} / \cong $$
where $(y_k) \cong (z_k) \mbox{if and only if } \frac{d(x_k,y_k)}{\lambda_k} \to 0$; equipped with a metric $$d([y_k],[z_k]) = \lim_{\alpha} \left( \frac{d(x_k,y_k)}{\lambda_k}\right)$$ where $[x_k]$ denotes the equivalent class of $(x_k)$. 
\end{definition}

When the context is clear we will denote (Cone$(X),d)$ simply as Cone$(X)$. An example is given by Cone$(\mathbb{R}^n) = \mathbb{R}^n$. 

In the case where $X$ is a symmetric space, it is known that Cone$(X)$ is independent of the ultrafilter \cite{Thornton} and has the structure of a Euclidean building \cite{Kleiner-Leeb}. As a building, Cone$(X)$ has a collection of apartments. An apartment of a building is a isometric embedding $F: \mathbb{R}^{k} \rightarrow \mbox{Cone}(X)$ satisfying some compatibility conditions. That is to say for each $x,y \in $ Cone$(X)$ there exists an isometric embedding $F: \mathbb{R}^{k} \rightarrow \mbox{Cone}(X)$ so that $x,y \in F(\mathbb{R}^k)$. 

In the case of $X = \mbox{SL}_n(\mathbb{C})/\mbox{SU}(n)$, the apartments are given by the limit of flats in SL$_n(\mathbb{C})/\mbox{SU}(n)$. In particular let $$A = \{ \vec{x} \in \mathbb{R}^n | \Sigma x_i = 0 \} \subset \mathbb{R}^n.$$ If we define $f: A \rightarrow \mbox{SL}_n(\mathbb{C})/\mbox{SU}(n)$ by $$f(\vec{x}) = \mbox{diag}(e^{x_1},\cdots,e^{x_n})$$ then it is well-known that $f$ parametrizes a flat in SL$_n(\mathbb{C})/\mbox{SU}(n)$. The map $f$ induces a map to Cone$(\mbox{SL}_n(\mathbb{C})/\mbox{SU}(n), Id, \{ \lambda_k \})$ in the following way: $$F: A \rightarrow \mbox{Cone}(\mbox{SL}_n(\mathbb{C})/\mbox{SU}(n))$$ $$\vec{x} \to F(\vec{x}) = [ f(\lambda_k\vec{x})].$$ One can show that F maps to an apartment in the cone. 

\section{Local results}

\label{lems}
In this section we will prove two lemmas on the local asymptotics of the solutions to the differential equation. Then in section \ref{PROOF} we will \textit{glue} together these results to prove Theorem \ref{thm1}. In the case where $(q_2,q_3,\cdots,q_n) = (0,0,\cdots,t\omega_n)$ for $\omega_n \in \mbox{H}^0(X,\mathcal{K}^n)$ where $\omega_n$ has simple zeros, equation \eqref{eq0} becomes
\begin{equation} \label{eq1} y^{(n)} + t\omega_n y = 0.
\end{equation} 

To study the asymptotics of the differential equation \eqref{eq1}, we use two model equations. For $z_0 \in \tilde{X}$ not a zero of $\omega_n$ the model equation is 
\begin{equation} \label{M1} y^{(n)} +ay = 0 \tag{M1} \end{equation}
where $a$ is a constant. Here the solutions are known explicitly. We first show in Lemma \ref{lem1} that near $z_0$, up to a change of coordinates, solutions to \eqref{eq1} become asymptotically solutions to \eqref{M1} the model equation. For the case $z_0$ a zero of $\omega_n$ the model equation is 
\begin{equation} \label{M2} y^{(n)}(\zeta) + \zeta^k y(\zeta) = 0 \tag{M2} \end{equation}
where $k$ is the order of the zero of $\omega_n$. Although we don't know explicitly the solutions to \eqref{M2} in Lemma \ref{lem2} we compute the asymptotics and identify the Stokes lines of \eqref{M2}. For $k = 1$ Wasow shows that, near $z_0$ a zero of $\omega_n$, solutions to \eqref{eq1} become asymptotically solutions to the model equation \eqref{M2} (see \cite{Wasow}).

Equation \eqref{eq1} is equivalent to 
\begin{equation} \label{eq2} \epsilon Y' = \left(\begin{array}{cccc} 0 & 1 & & \\ & 0 & 1 & \\ & & \ddots & 1 \\ -\omega_n & & & 0 \end{array} \right)Y = A(z)Y
\end{equation}
where $\epsilon = t^{-\frac{1}{n}}$ and $Y(z,t) = \mbox{diag}(1,\epsilon, \cdots, \epsilon^{n-1})W(y_1(z),\cdots,y_n(z))^T$.
\begin{lemma} \label{lem1}
If $z_0$ is not a zero of $\omega_n$, there is a open set $U_0$ containing $z_0$ such that there is a matrix solution $Y(z,t)$, with det$(Y(z,t)) \neq 0$ to \eqref{eq2} such that $$\lim_{t \to \infty}\left(Y(z,t)e^{-t^{\frac{1}{n}}\int B_0(z)}\right) < \infty$$  where $B_0(z) = diag(\omega_n^{\frac{1}{n}},\cdots,\omega_n^{\frac{1}{n}}\lambda^{n-1}) $ and $\lambda = e^{\frac{2\pi i}{n}}$. This shows that there exists a basis $\{ y_i(z) \}$ of solution \eqref{eq1} such that in the natural coordinates $\zeta$ we have $$y_i(\zeta) \sim e^{t^{\frac{1}{n}}\lambda^i \zeta}.$$
\end{lemma}

\begin{proof}
In a neighborhood around $z_0$, the matrix $A(z)$ has distinct eigenvalues. Then by Theorem 25.2 in \cite{Wasow}, there exists an open set $U$ containing $z_0$ and a transformation $Y = P(z,\epsilon)Z$ such that the differential equation \eqref{eq2} becomes 
\begin{equation} \label{eq3} \epsilon Z' = B(z,\epsilon)Z \end{equation} 
where $B(z,\epsilon)$ has asymptotic expansion $B(z,\epsilon) \sim \Sigma_{r=0}^{\infty} B_r(z)\epsilon^r$ with $B_r(z)$ diagonal and $B_0(z)$ = diag$(\lambda_1,...,\lambda_n)$; here $\lambda_i$ are the eigenvalues of $A(z)$. We have a matrix solution $Z(z)$ to \eqref{eq3} satisfying $$Z(z) = \hat{Z}(z,\epsilon)e^{\frac{1}{\epsilon}\int B_0(z)}e^{\int B_1(z)}$$ $\hat{Z}(z,\epsilon)$ has asymptotic expansion $\Sigma \hat{Z}_r(z) \epsilon^r$, with $\hat{Z}_0(z) = \mbox{I}_{n}$. Substituting this expression for $Z(z)$ yields a matrix solution $Y(z,\epsilon)$ to \eqref{eq2} where 
\begin{align} \label{EQ2} Y(z,\epsilon) &= P(z,\epsilon)Z(z) \\ \nonumber &= P(z,\epsilon)\hat{Z}(z,\epsilon)e^{\frac{1}{\epsilon}\int B_0(z)}e^{\int B_1(z)}. \end{align}

The leading term for $Y(z,\epsilon)$ are given in terms of $P_0$, $B_0$ and $B_1$. We know want to compute $P_0$, $B_0$ and $B_1$ in terms of $\omega_n$. If $P(z,\epsilon)$ has asymptotic expansion $P(z,\epsilon) \sim \Sigma P_r \epsilon^r$, the relationship between $Z$ and $Z'$ is $$\epsilon Z' = (P^{-1}AP - \epsilon P^{-1}P')Z = BZ.$$ The above equation relates the asymptotic expansions of $P(z,\epsilon)$ and $B(z,\epsilon)$ from which we obtain the equations 
\begin{equation} \label{Bzero} B_0 = P_0^{-1}AP_0 \end{equation}
and
\begin{equation} \label{Bone} B_1 = B_0 P_0^{-1}P_1 - P_0^{-1}P_1 B_0 - P_0^{-1}P_0'. \end{equation} 
Equation \eqref{Bzero} tells us that $P_0$ diagonalizes $B_0$ so that we can take  $$P_0 = \left(\begin{array}{ccc} 1 & \cdots & 1 \\ \omega_n^{\frac{1}{n}} & & \lambda^{(n-1)}\omega_n^{\frac{1}{n}} \\ \vdots & & \vdots \\ \omega_n^{\frac{(n-1)}{n}} & \cdots & \lambda \omega_n^{\frac{(n-1)}{n}} \end{array} \right) = \mbox{diag}(1,\omega_n^{\frac{1}{n}},\cdots,\omega_n^{\frac{(n-1)}{n}})\left(\begin{array}{ccc} 1 & \cdots & 1 \\ 1 & & \lambda^{(n-1)} \\ \vdots & & \vdots \\ 1 & \cdots & \lambda  \end{array} \right) .$$ We can now compute $B_1$ using equation \eqref{Bone}. If we let diag$(A$) denote the diagonal part of a matrix $A$ we find $$B_1 = \mbox{diag}(B_1) = \mbox{diag}(B_0 P_0^{-1}P_1 - P_0^{-1}P_1 B_0 - P_0^{-1}P_0')$$ because $B_1$ is diagonal. Thus $$ B_1 = \mbox{diag}(B_0 P_0^{-1}P_1 - P_0^{-1}P_1 B_0) - \mbox{diag}(P_0^{-1}P_0')$$ $$ = -\mbox{diag}(P_0^{-1}P_0').$$ the last step uses the linearity of diag($A$).
One can check that $\mbox{diag}(P_0^{-1}P_0') = \frac{n-1}{2n}\frac{\omega_n'}{\omega_n}\mbox{I}_{n} = -B_1$ 
\begin{equation}\label{EQ} e^{\int B_1(z)} = \omega_n^{\frac{1-n}{2n}}\mbox{I}_{n}. \end{equation}
Then (\ref{EQ2}) together with (\ref{EQ}) imply that we have the following asymptotic expansion for $Y(z,\epsilon)$; $$ Y(z,\epsilon)e^{-\frac{1}{\epsilon}\int B_0(z)} = P(z,\epsilon)\hat{Z}(z)\omega_n^{\frac{1-n}{2n}} = \Sigma \hat{Y}_r(z) \epsilon^r $$ with $\hat{Y}_0(z) = \omega_n^{\frac{1-n}{2n}} P_0Z_0 = \omega_n^{\frac{1-n}{2n}} P_0$. This tells us that \begin{equation}
\lim_{\epsilon \to 0}\left(Y(z,\epsilon)e^{-\frac{1}{\epsilon}\int B_0(z)}\right) = \omega_n^{\frac{1-n}{2n}}P_0.
\end{equation}
In terms of solutions to \eqref{eq1}, we have the following: there is a basis $y_i(z)$ of solutions such that $$y_i(z) \sim \omega_n^{\frac{1-n}{2n}} e^{\frac{\lambda^i}{\epsilon}\int \omega_n^{\frac{1}{n}}}.$$ After we change coordinates to $\zeta = \int \omega_n^{\frac{1}{n}}$, we have that $$y_i(\zeta) \sim e^{t^{\frac{1}{n}}\lambda^i \zeta}.$$
\end{proof}
The next lemma will compute the Stokes lines and asymptotics for the model equation \eqref{M2}. A theorem in \cite{wasow1963simplification} will show that in a neighborhood of a zero of $\omega_n$, there exists a transformation so that \eqref{eq1} to \eqref{M2}. 
\begin{lemma} \label{lem2}
For any sector $S$ centered at the orgin with central angle $\theta < \frac{n\pi}{(n+k)}$. There exists a basis of solutions $\{y_i(\zeta)\}$ to \begin{equation} \label{eq4} y^{(n)} - \zeta^k y = 0 \end{equation}
such that $y_i(\zeta) \sim \zeta^{\frac{k(1-n)}{2n}}e^{\lambda^i\frac{n \zeta^{\frac{n+k}{n}}}{n+k}}$ for $\zeta \in S$, and equation \eqref{eq4} has Stokes lines \newline $l_{ij} = \{ Re((\lambda^i -\lambda^j)\zeta^{\frac{n+k}{n}}) =0 \} $
\end{lemma}
\begin{proof}
The differential equation \eqref{eq4} is equivalent to the following system
\begin{equation*} \frac{dY}{d\zeta} = Y^\prime(\zeta) = \begin{pmatrix}
0 & 1 & & \\
& \ddots & \ddots & \\
& & \ddots & 1 \\
\zeta^k & & & 0 \\
\end{pmatrix} Y(\zeta).
\end{equation*}
We can rewrite the equation above as 
\begin{equation}
\label{EQ5}
\zeta^{-k}Y^\prime(\zeta) = \begin{pmatrix}
0 & \zeta^{-k} & & \\
& \ddots & \ddots & \\
& & \ddots & \zeta^{-k} \\
1 & & & 0 \\
\end{pmatrix} Y(\zeta).
\end{equation}
To compute the asymptotics we want to apply a theorem of Wasow (Theorem \ref{Wasow}). In order to apply the theorem, we will apply a transformation $Z(\zeta) = T(\zeta)Y(\zeta)$ so that $Z(\zeta)$ satisfies an equation of the form \begin{equation} \label{EQ6} \zeta^{-q} Z'(\zeta) = A(\zeta)Z(\zeta) 
\end{equation} where $A(\zeta) \sim \sum_{r=0}^\infty A_r \zeta^{-r}$ with $A_0$ having distinct eigenvalues. Then relating $Z(\zeta)$ back to $Y(\zeta)$ we find the asymptotics of $Y(\zeta)$. For convenience we first apply the transformation $Z_1 = \begin{pmatrix} & 1\\\mbox{I}_{n-1} & \\ \end{pmatrix} Y = WY$, so that the leading term is in Jordan canonical form and  the equation \eqref{EQ5} becomes
\begin{equation}
\label{Z1}
\zeta^{-k}Z_1^\prime(\zeta) = \begin{pmatrix}
0 & 1 & & &\\
& 0& \zeta^{-k} & & \\
& & \ddots & \ddots & \\
& & & \ddots& \zeta^{-k} \\
\zeta^{-k} & & & & 0 \\
\end{pmatrix} Z_1(\zeta).
\end{equation}
Here the leading coefficient matrix is nilpotent. We will apply the shearing transformation below multiple times until we obtain a leading coefficient matrix with distinct eigenvalues.  We begin with the following shearing transformation $$W_1 = S(\zeta)Z_1, \quad \mbox{where} \quad S(\zeta) = \begin{pmatrix} 1 & & & \\ & \zeta^{\frac{k}{n}} & & \\ & & \ddots & \\ & & & \zeta^{\frac{(n-1)k}{n}} \\ \end{pmatrix}$$ 
to obtain 
$$
\zeta^{-k}W_1^\prime(\zeta) = \Bigg[\begin{pmatrix}
0 & \zeta^{\frac{-k}{n}} & & &\\
& 0 & \zeta^{-k- \frac{k}{n}} & & \\
& & \ddots & \ddots & \\
& & & \ddots & \zeta^{-k-\frac{k}{n}} \\
\zeta^{\frac{-k}{n}} & & & & 0 \\
\end{pmatrix} + k\zeta^{-k-1}A(n)\Bigg]W_1(\zeta)
$$
where $A(n) = \frac{1}{n} \begin{pmatrix} 0 & & & \\ & 1 & & \\ & & \ddots & \\ & & & n-1 \\ \end{pmatrix}.$
Next re-normalize by multiplying both sides of the equation above by $\zeta^{\frac{k}{n}}$ to obtain
$$
\zeta^{-k+\frac{k}{n}}W_1^\prime(\zeta) = \Bigg[\begin{pmatrix}
0 & 1 & & &\\
& 0 & \zeta^{-k} & & \\
& & \ddots & \ddots & \\
& & & \ddots & \zeta^{-k} \\
1 & & & & 0 \\
\end{pmatrix} + k\zeta^{-k-1+\frac{k}{n}}A(n)\Bigg]W_1(\zeta).
$$ 
As before for convenience we apply the transformation $Z_2(\zeta) = \begin{pmatrix} & 1\\ \mbox{I}_{n-1} & \\ \end{pmatrix}W_1(\zeta)$ so that the leading term is in Jordan canonical form  to obtain 
\begin{equation}
\label{Z2}
\zeta^{-k+\frac{k}{n}}Z_2^\prime(\zeta) = \Bigg[\begin{pmatrix}
0 & 1 & & & &\\
& 0 & 1 & & &\\
& & 0 & \zeta^{-k} & & \\
& & &\ddots & \ddots & \\
& & & &\ddots & \zeta^{-k} \\
\zeta^{-k} & & & & &0 \\
\end{pmatrix} + k\zeta^{-k-1+\frac{k}{n}}\sigma(A(n))\Bigg]Z_2(\zeta),
\end{equation}
where $\sigma$ is conjugation by $W$. Comparing \eqref{Z1} and \eqref{Z2}, we find $Z_2 = \begin{pmatrix} & 1\\ I & \\ \end{pmatrix} S(\zeta) Z_1 = [WS(\zeta)]Z_1$ and the leading term  \begin{equation*}
\begin{pmatrix}
0 & 1 & & \\
& 0 & &\\
& &\ddots &\\
& & & 0 \\
\end{pmatrix}\end{equation*} 
has become
  \begin{equation*}
\begin{pmatrix}
0 & 1 & & \\
& 0 & 1 &\\
& &\ddots &\\
& & & 0 \\
\end{pmatrix}.
\end{equation*}

Each time we apply the above transformations the leading matrix gains a 1 in the upper diagonal. We do this until the leading matrix has distinct eigenvalues, which happens when $Z_n = [\begin{pmatrix} & 1\\ I & \\ \end{pmatrix}S(\zeta)]^{n-1}Z_1,$ resulting in the equation
\begin{equation} \label{eq7}
\zeta^{\frac{-k}{n}}Z_n^\prime(\zeta) = \Bigg[\begin{pmatrix}
0 & 1 & & \\
& \ddots & 1 & \\
& & \ddots & 1 \\
1 & & & 0 \\
\end{pmatrix} + k\zeta^{-1-\frac{k}{n}}\hat{A}(n)\Bigg]Z_n(\zeta),
\end{equation}
where $\hat{A}(n) = \sigma(A(n)) + \hdots + \sigma^{n-1}(A(n))$

For convenience we apply the transformation $Z = SZ_n$, where $$S\begin{pmatrix}
0 & 1 & & \\
& \ddots & 1 & \\
& & \ddots & 1 \\
1 & & & 0 \\
\end{pmatrix}S^{-1} = \begin{pmatrix} 1 & & & \\ & \lambda & & \\ & & \ddots & \\ & & & \lambda^{n-1} \\ \end{pmatrix} = D,$$ then equation \eqref{eq7} becomes
$$
\zeta^{\frac{-k}{n}}Z^\prime(\zeta) = [D + k\zeta^{-1-\frac{k}{n}}S\hat{A}(n)S^{-1}]Z(\zeta)
.$$
Now substituting $\zeta = \alpha t^n$ with $\alpha = n^{-\frac{n}{n+k}}$ yields

$$
t^{-(n+k-1)}Z^\prime(t) = [D
+ kt^{-n-k}nS\hat{A}(n)S^{-1}]Z(t).
$$
We now have an equation of the form \eqref{EQ6} where the leading matrix $A_0$ has distinct eigenvalues and we can apply Theorem \ref{Wasow}. This implies that for any central sector with angle $< \frac{\pi}{q+1}$, where $q = n + k-1$, there is a solution $Z(t)$ to the equation above such that $$Z(t) = \hat{Z}(t)e^{Q(t)}t^{\Lambda_0},$$
where $\hat{Z}(t) = \Sigma Z_r t^{-r}$ with $Z_0 = \mbox{I}_{n}.$

In this case $Q(t)$ and $\Lambda_0$ have the forms, $$Q(t) = \begin{pmatrix} 1 & & & \\ & \lambda & & \\ & & \ddots & \\ & & & \lambda^{n-1} \\ \end{pmatrix} \frac{t^{n+k}}{n+k} = D \frac{t^{n+k}}{n+k}$$
$$\mbox{and} \quad \Lambda_0 = \frac{k(n-1)^2}{2} \mbox{I}_{n}. $$
Of course, we are interested in the asymptotics for a solution $Y(\zeta)$ to the original equation \eqref{EQ5}.  To find this we can relate $Z(\zeta)$ to $Y(\zeta)$ by the following equations: 
$Z = SZ_n = S[WS(\zeta)]^{n-1}Z_1 = S[WS(\zeta)]^{n-1}WY $ or $Y(\zeta) = T(\zeta)Z(\alpha^{\frac{-1}{n}}\zeta^{\frac{1}{n}})$, where \begin{equation} \label{EQ3} T(\zeta)^{-1} = S\bigg[\begin{pmatrix} & 1\\ \mbox{I}_{n-1} & \\ \end{pmatrix}S(\zeta)\bigg]^{n-1}\begin{pmatrix} & 1\\ \mbox{I}_{n-1} & \\ \end{pmatrix}. \end{equation} 
So we have $Y(\zeta) = T(\zeta)\hat{Z}(\alpha^{\frac{-1}{n}}\zeta^{\frac{1}{n}}) e^{D \frac{n \zeta^{\frac{n+k}{n}}}{n+k}}\zeta^{\frac{(n-1)^2k}{2n}}.$

We can simpliy the expression for $T(\zeta)$ by noting $$[\begin{pmatrix} & 1\\ \mbox{I}_{n-1} & \\ \end{pmatrix}S(\zeta)]^n = \prod_{i=0}^{n-1} \zeta^{\frac{ik}{n}}\mbox{I}_{n},$$ so that \begin{equation} \label{EQ4} \bigg[\begin{pmatrix} & 1\\ \mbox{I}_{n-1} & \\ \end{pmatrix}S(\zeta)\bigg]^{n-1} = \prod_{i=0}^{n-1} \zeta^{\frac{ik}{n}}S^{-1}(\zeta)\begin{pmatrix} & 1\\ \mbox{I}_{n-1}& \\ \end{pmatrix}^{-1} \end{equation}
Equations \eqref{EQ3} and \eqref{EQ4} together imply $$T(\zeta)^{-1} = S\prod_{i=0}^{n-1} \zeta^{\frac{ik}{n}}S^{-1}(\zeta)\begin{pmatrix} & 1\\ \mbox{I}_{n-1} & \\ \end{pmatrix}^{-1}\begin{pmatrix} & 1\\ \mbox{I}_{n-1} & \\ \end{pmatrix} = SS^{-1}(\zeta)\zeta^{\frac{k(n-1)}{2}},$$ 
which gives us 
\begin{align*}Y(\zeta) &= S(\zeta)S^{-1}\zeta^{\frac{k(1-n)}{2}}\hat{Z}(\alpha^{\frac{-1}{n}}\zeta^{\frac{1}{n}}) e^{D \frac{n \zeta^{\frac{n+k}{n}}}{n+k}}\zeta^{\frac{(n-1)^2k}{2n}} \\ &= 
S(\zeta)S^{-1}\hat{Z}(\alpha^{\frac{-1}{n}}\zeta^{\frac{1}{n}})e^{D \frac{n \zeta^{\frac{n+k}{n}}}{n+k}}\zeta^{\frac{k(1-n)}{2n}}. \end{align*}

So we find there is a basis $\{ y_i(\zeta)\}$ of solutions to (2) so that $y_i(\zeta) \sim \zeta^{\frac{k(1-n)}{2n}}e^{\lambda^i\frac{n \zeta^{\frac{n+k}{n}}}{n+k}}$ for  $\zeta \in S$ any central sector with angle $ < \frac{n\pi}{q+1} = \frac{n\pi}{n+k}$ It then follows from Wasow \cite{Wasow} that the Stokes lines for these equations are at the rays where $l_{ij} = \{ Re((\lambda^i -\lambda^j)\zeta^{\frac{n+k}{n}}) =0 \} $ for some $i \neq j$ with Stokes matrix $A = \mbox{I}_{n}+aE_{ij}$ for some $a\in \mathbb{C}.$

\end{proof} 
We now return our attention back to equation \eqref{eq2}. The previous lemma and an application of \cite{wasow1963simplification} allows us to compute the asymptotics of \eqref{eq2} in a neighborhood of a zero of $\omega_n$. 
\begin{lemma} \label{lem3}
Let $z$ be a zero of $\omega_n$, then there is a open set $V$ containing $z$ such that there is a matrix solution $Y(z,t)$, with det$(Y(z,t)) \neq 0$, to \eqref{eq2} such that $$\lim_{t \to \infty}\left(Y(z,t)e^{-t^{\frac{1}{n}}\int B_0(z)}\right) < \infty$$ for z lying an a sector with $\zeta$-angle $< \frac{n\pi}{(n+1)}$ and has the same Stokes data as \ref{lem2}.
\end{lemma}
\begin{proof}
In $\zeta$ the natural coordinates of $\omega_n$, in a neighborhood of $z$ we have $\omega_n(\zeta) = \zeta d\zeta^n$. In these coordinates \eqref{eq1} becomes 
\begin{equation}
y^{(n)}+t\zeta y + \sum_{k=2}^{n-1} Q_k(\zeta)y^{(n-k)} = 0
\end{equation}
This is equivalent to 
\begin{equation} \epsilon Y' = \left[ \left(\begin{array}{cccc} 0 & 1 & & \\ & 0 & 1 & \\ & & \ddots & 1 \\ -\zeta & & & 0 \end{array} \right) + \epsilon^n\left(\begin{array}{cccc} 0 & 1 & & \\ & 0 & 1 & \\ & & \ddots & 1 \\ 0 & Q_{n-1}(\zeta) & \cdots & 0 \end{array} \right) \right]Y = [A_0(\zeta)+ \epsilon^n A_n(\zeta) ] Y
\end{equation}
where $Y(z,t) = \left(\begin{array}{c} y \\ \epsilon y' \\ \vdots \\ \epsilon^{n-1} y^{(n-1)} \end{array} \right),$ and $\epsilon = t^{-\frac{1}{n}}$. Now by applying a theorem in \cite{wasow1963simplification} we have that there is a $T(\zeta,t)$, such that $\lim_{t \to \infty} T(\zeta,t) = T_0$ and det($T_0) \neq 0$, so that if we let $Z(\zeta,t) = T(\zeta,t)Y(\zeta,t)$ we obtain 
\begin{equation} 
\label{W1} \epsilon Z' = \left(\begin{array}{cccc} 0 & 1 & & \\ & 0 & 1 & \\ & & \ddots & 1 \\ -\zeta & & & 0 \end{array} \right)Z = A_0(\zeta)Z.
\end{equation}
Let $Y_*(z)$ be a solution to the equation

$$Y_{*} '(\zeta) = \begin{pmatrix}
0 & 1 & & \\
& \ddots & \ddots & \\
& & \ddots & 1 \\
\zeta & & & 0 \\
\end{pmatrix} Y_*(\zeta).$$

Then a solution to equation \eqref{W1} is given by 
$$ Z(\zeta,t) = t^{\frac{n-1}{2n(n+1)}}\mbox{diag}(1,t^{-\frac{1}{n(n+1)}},\cdots ,t^{-\frac{n-1}{n(n+1)}})Y_*(\zeta t^{\frac{1}{n+1}}) $$
Lemma \ref{lem2} tells us that for $\zeta$ lying an a sector with angle $< \frac{n\pi}{(n+1)}$ we have 
$$Y_*(\zeta) = S(\zeta)S^{-1}\hat{Z}(\alpha^{\frac{-1}{n}}\zeta^{\frac{1}{n}})e^{D \frac{n \zeta^{\frac{n+1}{n}}}{n+1}}\zeta^{\frac{(1-n)}{2n}}.$$
This implies that 
\begin{align*}
 Z(\zeta,t)e^{-D \frac{n (\zeta t^{\frac{1}{n+1}})^{\frac{n+1}{n}}}{n+1}} &= S(\zeta)S^{-1}\hat{Z}(\alpha^{\frac{-1}{n}}(\zeta t^{\frac{1}{n+1}})^{\frac{1}{n}})\zeta^{\frac{(1-n)}{2n}} 
 \\  T(\zeta,t)Y(\zeta,t)e^{-D \frac{n (\zeta t^{\frac{1}{n+1}})^{\frac{n+1}{n}}}{n+1}} &= S(\zeta)S^{-1}\hat{Z}(\alpha^{\frac{-1}{n}}(\zeta t^{\frac{1}{n+1}})^{\frac{1}{n}})\zeta^{\frac{(1-n)}{2n}} \\ Y(\zeta,t)e^{-D \frac{n (\zeta t^{\frac{1}{n+1}})^{\frac{n+1}{n}}}{n+1}} &= T^{-1}(\zeta,t)S(\zeta)S^{-1}\hat{Z}(\alpha^{\frac{-1}{n}}(\zeta t^{\frac{1}{n+1}})^{\frac{1}{n}})\zeta^{\frac{(1-n)}{2n}}.
\end{align*}
Changing to $z$-coordinates and taking limits give us 
$$ \lim_{t \to \infty} Y(z,t)e^{-\frac{1}{\epsilon}\int B_0(z)} =$$ $$  \lim_{t \to \infty} T^{-1}(\zeta(z),t)S(\zeta(z))S^{-1}\hat{Z}(\alpha^{\frac{-1}{n}}(\zeta(z) t^{\frac{1}{n+1}})^{\frac{1}{n}})\zeta(z)^{\frac{(1-n)}{2n}} < \infty. 
$$
for z lying an a sector with $\zeta$-angle $< \frac{n\pi}{(n+1)}$ and we see that $Y(z,t)$ has the same Stokes data as $Y_*(\zeta)$.
\end{proof}
Lemma \ref{lem1} computes the asymptotics of the differential equation \eqref{eq1} away from a zero of $\omega_n$, and lemma \ref{lem3} computes the asymptotics near a zero. At any point $p \in \tilde{X}$ we can understand the local behavior. In the next section we use these results to prove \ref{thm1}.

\section{Proof of Theorem \ref{thm1}}

\label{PROOF}
The goal of this section is to prove the following theorem. 
\begin{theorem*}

Let $\rho_t$ be the holonomy associated to $(0,\cdots,t\omega_n)$ by the map above and $[\gamma] \in \pi_1(X)$. Then there exists a decomposition of $\gamma = \gamma_1 +\cdots+ \gamma_N$ and a collections of Stokes matrices $A_i$ so that $$\lim_{t \to \infty}\rho_t([\gamma])( \prod_{i=1}^{N} E(\gamma_{i} ,\omega_{n} )A_i)^{-1} = I_n$$
\end{theorem*}
\begin{proof}

We begin by outlining our approach. Fix $[\gamma] \in \pi_1(X)$ and $z_0 \in \tilde{X}$ which is not a zero of $\omega_n$, let $\rho_{\epsilon}$ be the holonomy associated to $t \omega_n$ and let $\gamma$ be the $|\omega_n|^{\frac{1}{n}}$ metric geodesic representative of $[\gamma]$ in $\tilde{X}$ connecting $z_0$ and $\gamma(z_0)$. For each $x \in \gamma$ there are two cases. If $x \in \gamma$ is not a zero of $\omega_n$ lemma \ref{lem1} implies that we have a neighborhood $U_x$, where we can compute the asymptotics. For $x$ which is a zero of $\omega_n$ lemma \ref{lem3} implies that there exists a neighborhood $V_x$ where we can compute the asymptotics. The sets $\{U_x , V_x \}$ results in a open cover of $\gamma$. By compactness we can pass to a finite sub-cover $\{ U_k \}$ where, on each open set $U_k$, there is a matrix solution $Y_k$ for which we can compute the asymptotics. When two such sets $U_k$, $U_{k'}$ intersect we can relate the corresponding solutions $Y_k$ and $Y_{k'}$ by Stokes matrices. By iteratively relating solutions on consecutive elements of the cover $\{ U_i \}$ of $\gamma$, we can see how the solutions develops along $\gamma$. In particular, this allows us to relate the matrix solution $Y_0$, the solution corresponding to $U_0$ the open set containing $z_0$, and the matrix solution $Y_1$, the solution corresponding to $U_\gamma$ the open set containing $\gamma(z_0)$.

Recall in section \ref{bg} we saw that $\rho_{\epsilon}([\gamma])$ is determined by 
\begin{equation} \label{hol}\rho_\epsilon([\gamma]) = Y_0(z,\epsilon)^{-1}Y_\gamma(z,\epsilon)
\end{equation} where $Y_0(z,\epsilon)$ is the matrix solution  corresponding to a basis of solutions $$y_1(z),\cdots, y_n(z)$$ and $Y_\gamma(z,\epsilon)$ is the matrix solution  corresponding to a basis of solutions $$y_1(\gamma(z))\gamma'^{\frac{1-n}{2}},  y_2(\gamma(z))\gamma'^{\frac{1-n}{2}},\cdots, y_n(\gamma(z))\gamma'^{\frac{1-n}{2}}.$$
More precisely we have 
\begin{align*}Y_0(z,\epsilon) &= \mbox{diag}(1,\epsilon, \cdots, \epsilon^{n-1}) W(y_1(z),y_2(z),\cdots,y_n(z))^T, 
\end{align*}  and 
\begin{align*}Y_\gamma(z,\epsilon) &= \mbox{diag}(1,\epsilon, \cdots, \epsilon^{n-1}) W(y_1(\gamma(z))\gamma'^{\frac{1-n}{2}},y_2(\gamma(z))\gamma'^{\frac{1-n}{2}},...,y_n(\gamma(z))\gamma'^{\frac{1-n}{2}})^T \\ &= N(\gamma,\epsilon)Y_0(\gamma(z),\epsilon)
\end{align*} 
where $ N(\gamma,\epsilon) = \mbox{diag}(1,\epsilon, \cdots, \epsilon^{n-1})M_1(\gamma'^{\frac{1-n}{2}})^T M_2(\gamma')^T\mbox{diag}(1,\epsilon, \cdots, \epsilon^{n-1})^{-1}$ and $M_1, M_2$ are defined in section \ref{bg}.

Knowing asymptotics of $Y_0(z,\epsilon)$ and $Y_0(\gamma(z),\epsilon)$ will give us asymptotics for $\rho_\epsilon([\gamma])$, using \eqref{hol}. More precisely, for each $x \in \gamma$ which is not a zero of $\omega_n$ we can apply Lemma \ref{lem1} to we see that there exists an open set $U_x$ containing $x$ and $Y_x(z,\epsilon)$ a matrix solution so that $$ \lim_{\epsilon \to 0}\left(Y_x(z,\epsilon)e^{-\frac{1}{\epsilon}\int_{x}^z B_0(z)}\right) =\omega_n^{\frac{1-n}{2n}}P_0 \mbox{ with } z \in U_x.$$ Next, for each point $p_j$ on $\gamma$ which is a zero of $\omega_n$ we can apply Lemma \ref{lem3} to see that there exists a neighborhood $V_j$ of $p_j$ and finitely many sectors $S_{ji}$ of $V_j$ each with $Y_{ji}(z,\epsilon)$ a matrix solution so that $$ \lim_{\epsilon \to 0}\left(Y_{ji}(z,\epsilon)e^{-\frac{1}{\epsilon}\int_{x}^z B_0(z)}\right) =\omega_n^{\frac{1-n}{2n}}P_0 \mbox{ with } z \in S_{ji}.$$

These sets $\{ U_x, V_j \}$ form an open cover of $\gamma$. Since $\gamma$ is compact we can pass to a finite subcover, $V_j$, $ 1 \leq j \leq N_1$ and  $U_{i}$, $ 0 \leq i \leq N_2$. Up to reordering and shrinking we can label the cover $\{ U_{i}, V_j \}$ so that if we parametrize $\gamma: [0,1] \rightarrow \tilde{X}$ there exists a partition $t_0 = 0, \cdots ,t_{N} = 1$ satisfying $$\gamma((t_i,t_{i+1})) \subset U_{i}$$ or $$\gamma((t_i,t_{i+1})) \subset V_i$$

For each open set $\{ U_{i} \}$ there is a corresponding matrix solution $Y_i(z,\epsilon)$ coming from Lemma \ref{lem1} so that 
\begin{equation} \label{above}
\lim_{\epsilon \to 0}\left(Y_i(z,\epsilon)e^{-\frac{1}{\epsilon}\int_{x_i}^z B_0(z)}\right) =\omega_n^{\frac{1-n}{2n}}P_0 \mbox{ with } z \in U_{i}.
\end{equation}
For each k we want to relate $Y_k$ to $Y_{k+1}$ by showing that if $Y_k$ satisfies \eqref{above} for $i = k$ then there exists a matrix of the form $A_{q_k} = e^{-\frac{1}{\epsilon}\int_{x_k}^{q_k} B_0(t)dt}A_{\alpha \beta}e^{-\frac{1}{\epsilon}\int_{q_k}^{x_{k+1}}B_0(t)dt}$ so that $Y_{k+1}(z,\epsilon) = Y_k(z,\epsilon)A_q$ satisfies \eqref{above} for $i = k+1$. This implies that \begin{equation} \label{stepN}
Y_N(z,\epsilon) = Y_0(z,\epsilon) \prod_k A_{q_k} 
\end{equation}
 satisfies \eqref{above} for $i = N$. Then equations \eqref{stepN} and \eqref{hol} together tells us that 
 \begin{align*}
\rho_\epsilon([\gamma]) &= Y_0(z,\epsilon)^{-1}Y_\gamma(z,\epsilon) \\
\rho_\epsilon([\gamma]) &= Y_0(z,\epsilon)^{-1}N(\gamma,\epsilon)Y_0(\gamma(z),\epsilon) \\
\rho_\epsilon([\gamma]) \prod_k A_{q_k}  &= Y_0(z,\epsilon)^{-1}N(\gamma,\epsilon)Y_N(\gamma(z),\epsilon).
 \end{align*}
Finally using equation \eqref{above}  we can compute $$ \lim_{\epsilon \to 0}\rho_\epsilon([\gamma]) \prod_k A_{q_k}.$$
To relate $Y_k$ and $Y_{k+1}$ we see that for each $k$ we have either $U_k \cap U_{k+1} \neq \emptyset$ or there exists a set  $V_{k+1}$ so that $U_k \cap V_{k+1} \neq \emptyset$ and $U_{k+2} \cap V_{k+1} \neq \emptyset$. In the case $U_k \cap U_{k+1} \neq \emptyset$ and the sets are separated by a Stokes line we have the following lemma.  
\begin{lemma} \label{StokesLEM}
If $U_k$ and $U_{k+1}$ are separated by a Stokes line $l_{\alpha \beta}$ coming from a zero $q$ and $Y_k(z,\epsilon)$ satisfies $$ \lim_{\epsilon \to 0}\left(Y_k(z,\epsilon)e^{-\frac{1}{\epsilon}\int_{x_k}^z B_0(z)}\right) =\omega_n^{\frac{1-n}{2n}}P_0 \mbox{ with } z \in U_{k}$$  then there exists a Stokes matrix $A_{\alpha \beta}$ so that $Y_{k+1}(z,\epsilon) = Y_k(z,\epsilon)A_q$ satisfies $$ \lim_{\epsilon \to 0}\left(Y_{k+1}(z,\epsilon)e^{-\frac{1}{\epsilon}\int_{x_{k+1}}^z B_0(z)}\right) =\omega_n^{\frac{1-n}{2n}}P_0 \mbox{ with } z \in U_{k+1}$$ where $A_q = e^{-\frac{1}{\epsilon}\int_{x_k}^q B_0(t)dt}A_{\alpha \beta}e^{-\frac{1}{\epsilon}\int_{q}^{x_{k+1}}B_0(t)dt} $
\end{lemma}
\begin{proof}
First note that $Y_q = Y_{k}(z,\epsilon)e^{\frac{1}{\epsilon}\int_q^{x_k}B_1(t)dt} $ satisfies the following $$\lim_{\epsilon \to 0}\left(Y_q(z,\epsilon)e^{-\frac{1}{\epsilon}\int_{q}^z B_0(z)}\right) =\omega_n^{\frac{1-n}{2n}}P_0 \mbox{ with } z \in U_k.$$ Then there is a Stokes matrix  $A_{\alpha \beta} = \mbox{I}_n + a E_{\alpha \beta}$ associated to the Stokes line $l_{\alpha \beta}$ so that $$ \lim_{\epsilon \to 0}\left(Y_q(z,\epsilon)A_{\alpha \beta}e^{-\frac{1}{\epsilon}\int_{q}^z B_0(z)}\right) =\omega_n^{\frac{1-n}{2n}}P_0 \mbox{ with } z \in U_{k+1}.$$ This implies that $Y_q A_{\alpha \beta}e^{-\frac{1}{\epsilon}\int_q^{x_{k+1}}B_0(t)dt}$ satisfies the following $$\lim_{\epsilon \to 0}\left(\left[Y_q A_{\alpha \beta}e^{-\frac{1}{\epsilon}\int_q^{x_{k+1}}B_0(t)dt}\right]e^{-\frac{1}{\epsilon}\int_{x_{k+1}}^z B_0(z)}\right) =\omega_n^{\frac{1-n}{2n}}P_0 \quad \mbox{ with } z \in U_{k+1}.$$ If we define $A_q = e^{-\frac{1}{\epsilon}\int_{x_k}^q B_0(t)dt}A_{\alpha \beta}e^{-\frac{1}{\epsilon}\int_{q}^{x_{k+1}}B_0(t)dt}$ then this shows that $$\lim_{\epsilon \to 0}\left(\left[Y_k(z,\epsilon)A_q\right]e^{-\frac{1}{\epsilon}\int_{x_{k+1}}^z B_0(z)}\right) =\omega_n^{\frac{1-n}{2n}}P_0 \quad \mbox{ with } z \in U_{k+1}.$$ Therefore we can take $Y_{k+1}(z,\epsilon) = Y_k A_q$.
\end{proof}
Thus, if $U_k$ and $U_{k+1}$ are separated by another Stokes line $l_{\alpha' \beta'}$ coming from a zero $q'$, with Stokes matrix $A'$, then by the same argument above we can say that  $$\lim_{\epsilon \to 0}\left(\left[Y_k(z,\epsilon)A'_{q'}\right]e^{-\frac{1}{\epsilon}\int_{x_{k+1}}^z B_0(z)}\right) =\omega_n^{\frac{1-n}{2n}}P_0 \quad \mbox{ with } z \in U_{k+1}$$ where $A'_{q'} = e^{-\frac{1}{\epsilon}\int_{x_k}^{q'} B_0(t)dt}A'e^{-\frac{1}{\epsilon}\int_{q'}^{x_{k+1}}B_0(t)dt}$.
A calculation shows that $$\lim_{\epsilon \to 0}  e^{\frac{1}{\epsilon}\int_{x_{k+1}}^z B_0(z)}(A'_{q'})^{-1}A_q e^{-\frac{1}{\epsilon}\int_{x_{k+1}}^z B_0(z)} = I_{n} \quad z \in U_{k+1}$$
so that $Y_k(z,\epsilon) A_q$ and $Y_k(z,\epsilon) A'_{q'}$ have the same asymptotics in $\epsilon$ for $z \in U_{k+1}.$

Another possibility is if the sets $U_k$ and $U_{k+1}$ are separated by a secondary Stokes line. When two Stokes lines $l_{\alpha \beta}$ and $l_{\beta \delta}$, with Stokes matrices $A$ and $B$ respectfully, cross they give birth to a new Stokes line $l_{\alpha \delta}$ \cite{berk1982new}. The Stokes line $l_{\alpha \delta}$ starts at the intersection point $x \in l_{\alpha \beta} \cap l_{\beta \delta}$ is defined $$ l_{\alpha \delta} =  \{ Re(\int_x^z \lambda_\alpha - \lambda_\delta) = 0 \} $$ This occurs because  two matrices $A$ and $B$ may not commute.
\begin{lemma}
If $U_k$ and $U_{k+1}$ are separated by a secondary Stokes line $l_{\alpha \delta}$, coming from the crossing of two Stokes line $l_{\alpha \beta}$ and $l_{\beta \delta}$ coming from two zeros $p, q$ and 
 $Y_k(z,\epsilon)$ satisfies $$ \lim_{\epsilon \to 0}\left(Y_k(z,\epsilon)e^{-\frac{1}{\epsilon}\int_{x_k}^z B_0(z)}\right) =\omega_n^{\frac{1-n}{2n}}P_0 \mbox{ with } z \in U_{k}$$  then there exists a Stokes matrix $C_{\alpha \delta}$ so that $Y_{k+1}(z,\epsilon) = Y_k(z,\epsilon)C_q$ satisfies $$ \lim_{\epsilon \to 0}\left(Y_{k+1}(z,\epsilon)e^{-\frac{1}{\epsilon}\int_{x_{k+1}}^z B_0(z)}\right) =\omega_n^{\frac{1-n}{2n}}P_0 \mbox{ with } z \in U_{k+1}$$ where $C_q = e^{-\frac{1}{\epsilon}\int_{x_k}^q B_0(t)dt}C_{\alpha \beta}e^{-\frac{1}{\epsilon}\int_{q}^{x_{k+1}}B_0(t)dt} $
\end{lemma}
\begin{proof}
Suppose two sets $U_m$ and $U_{m'}$ are separated by two Stokes lines  $l_{\alpha \beta}$ and $l_{\beta \delta}$. If $\gamma$ crosses $l_{\alpha \beta}$ first and then $l_{\beta \delta}$ combining the results from above we can relate 
\begin{equation} Y_m(z,\epsilon) = Y_{m'}(z,\epsilon)B_q A_p.
\end{equation}
On the other hand, if $\gamma$ crosses  $l_{\beta \delta}$ first and then $l_{\alpha \beta}$ combining the results from above we can relate \begin{equation}
Y_m(z,\epsilon) = Y_{m'}(z,\epsilon)A_p B_q.
\end{equation}
In order for these two representations to be equal we see we need a secondary Stokes matrix when crossing the $l_{\alpha \delta}$ line. If $U_k$ and $U_{k+1}$ are separated by a secondary Stokes line coming from the crossing of two Stokes line $l_{\alpha \beta}$ and $l_{\beta \delta}$ coming from two zeros $p, q$ then we can relate $Y_k(z,\epsilon)$ to $Y_{k+1}(z,\epsilon)$ in the following way 

$$ Y_{k+1}(z,\epsilon) = Y_k(z,\epsilon)C = Y_k(z,\epsilon)e^{\frac{1}{\epsilon}\int_{p}^{x_k}B_0(t)dt}C_{\alpha \delta}e^{-\frac{1}{\epsilon}\int_{p}^{x_{k+1}}B_0(t)dt}$$
where $C_{\alpha \delta} = \mbox{I}_n + a e^{\frac{1}{\epsilon}\int_q^p \lambda_\beta-\lambda_\delta} E_{\alpha \delta}$ is associated to the Stokes line $l_{\alpha \delta}$.
If $\gamma$ crosses $l_{\alpha \beta}$ first then $l_{\beta \delta}$ it will cross $l_{\alpha \delta}$ so we get that $$Y_m(z,\epsilon) = Y_{m'}(z,\epsilon)B_q C A_p$$ which one can check is equal to $Y_m(z,\epsilon) = Y_{m'}(z,\epsilon)A_p B_q.$

\end{proof}
The last case involves crossing a zero of $\omega_n$. In that case we have the following:
\begin{lemma}
If $U_k$ and $U_{k+2}$ are separated by an open set $V_{k+1}$ containing a zero $q$ of $\omega_n$  and 
 $Y_k(z,\epsilon)$ satisfies $$ \lim_{\epsilon \to 0}\left(Y_k(z,\epsilon)e^{-\frac{1}{\epsilon}\int_{x_k}^z B_0(z)}\right) =\omega_n^{\frac{1-n}{2n}}P_0 \mbox{ with } z \in U_{k}$$  then there exists a Stokes matrix $A$ so that $Y_{k+2}(z,\epsilon) = Y_k(z,\epsilon)A_q$ satisfies $$\lim_{\epsilon \to 0}\left(Y_{k+1}(z,\epsilon)e^{-\frac{1}{\epsilon}\int_{x_{k+1}}^z B_0(z)}\right) =\omega_n^{\frac{1-n}{2n}}P_0 \mbox{ with } z \in U_{k+1}$$ where $A_q = e^{-\frac{1}{\epsilon}\int_{x_k}^q B_0(t)dt}Ae^{-\frac{1}{\epsilon}\int_{q}^{x_{k+1}}B_0(t)dt} $
\end{lemma}

\begin{proof}
The open set $V_{k+1}$ is divided into finitely many sectors $S_{k+1,i}$, and $U_k$ intersects $V_{k+1}$ in a sector say $S_{k+1,i}$ and $U_{k+2}$ intersects $V_{k+1}$ in a sector say $S_{k+1,i'}$. Let $A$ be the Stokes matrix relating $S_{k+1,i}$ and $S_{k+1,i'}$, then by the same argument in Lemma \ref{StokesLEM} we have
$$ Y_{k+1}(z,\epsilon) = Y_k(z,\epsilon)e^{-\frac{1}{\epsilon}\int_{p_j}^{z_k}B_0(t)dt}Ae^{-\frac{1}{\epsilon}\int_{p_j}^{z_{k+1}}B_0(t)dt},$$ 

\end{proof}
If $U_k$ and $U_{k+1}$ are not separated by any Stokes lines, then there is a chart $V$, around a zero $p$ of $\omega_n$, so that $U_k$ and $U_{k+1}$ are contained in the same sector $S_i$. Since $Y_k(z,\epsilon)$ satisfies $$ \lim_{\epsilon \to 0}\left(Y_k(z,\epsilon)e^{-\frac{1}{\epsilon}\int_{x_k}^z B_0(z)}\right) =\omega_n^{\frac{1-n}{2n}}P_0 \mbox{ with } z \in U_{k},$$ and $U_k \cup U_{k+1} \subset S_i$, then $Y_k(z,\epsilon)$ satisfies $$\lim_{\epsilon \to 0}\left(Y_k(z,\epsilon)e^{-\frac{1}{\epsilon}\int_{x_k}^z B_0(z)}\right) =\omega_n^{\frac{1-n}{2n}}P_0 \mbox{ with } z \in U_{k+1}.$$
In each case above we have related the matrix solutions $Y_k$ and $Y_{k+1}$ for each $k$. Applying these lemmas successively we find that we can relate $Y_N$ and $Y_0$ in the following way, $$ Y_N(z,\epsilon) = Y_0(z,\epsilon)M(\epsilon),$$
where if we write $E_t(\gamma,\phi) = e^{-t^{\frac{1}{n}} \int_{\gamma} B_0(t)dt}$, then
\begin{align*}
M(\epsilon) & = e^{\frac{1}{\epsilon}\int_{q_1}^{z_0}B_0(t)dt}A_1e^{-\frac{1}{\epsilon}\int_{q_1}^{z_1}B_0(t)dt}\cdots e^{\frac{1}{\epsilon}\int_{q_{N-1}}^{z_{N}}B_0(t)dt}A_Ne^{-\frac{1}{\epsilon}\int_{q_N}^{z_N}B_0(t)dt}\\  & = e^{\frac{1}{\epsilon}\int_{q_1}^{z_0}B_0(t)dt}A_1e^{-\frac{1}{\epsilon}\int_{q_1}^{q_2}B_0(t)dt}A_2 \cdots A_{N-1}e^{\frac{1}{\epsilon}\int_{q_{N-1}}^{q_{N}}B_0(t)dt}A_Ne^{-\frac{1}{\epsilon}\int_{q_N}^{z_N}B_0(t)dt} \\&= \prod_{i=1}^{N+1} E(\gamma_{i} ,\omega_{n} )A_i
\end{align*}

The matrix solutions $Y_0$ and $Y_N$ satisfy the following: 

\begin{equation} \label{e1} \lim_{\epsilon \to 0}\left(Y_0(z,\epsilon)e^{-\frac{1}{\epsilon}\int_{z_0}^z B_0(z)}\right) =\omega_n^{\frac{1-n}{2n}}P_0 \mbox{ with } z \in U_0 \end{equation}

\begin{equation} \label{e2} \lim_{\epsilon \to 0}\left(Y_N(z,\epsilon)e^{-\frac{1}{\epsilon}\int_{z_N}^z B_0(z)}\right) =\omega_n^{\frac{1-n}{2n}}P_0 \mbox{ with } z \in U_N. \end{equation}

Using \eqref{e1}, \eqref{e2} and $Y_N(z,\epsilon) = Y_0(z,\epsilon)M(\epsilon)$ we know the asymptotics of $Y_0$ for $z \in U_0$ and for $z \in U_N$. Recalling \eqref{hol} we see that 

\begin{align*}
\rho_\epsilon(\gamma) &= Y_0(z,\epsilon)^{-1}Y_\gamma(z,\epsilon) \\
\rho_\epsilon(\gamma) &= Y_0(z,\epsilon)^{-1}N(\gamma,\epsilon)Y_0(\gamma(z),\epsilon) \\
\rho_\epsilon(\gamma)M(\epsilon)^{-1} &= Y_0(z,\epsilon)^{-1}N(\gamma,\epsilon)Y_0(\gamma(z),\epsilon)M(\epsilon)^{-1}. 
\end{align*}
Recall that $ N(\gamma,\epsilon) = \mbox{diag}(1,\epsilon, \cdots, \epsilon^{n-1})M_1(\gamma'^{\frac{1-n}{2}})^T M_2(\gamma')^T\mbox{diag}(1,\epsilon, \cdots, \epsilon^{n-1})^{-1}$ and the matrices $M_1, M_2$ are defined in section \ref{bg}. Using proposition \ref{induct} with $g = \gamma'^{\frac{1-n}{2}}$ and $h = \gamma$ we can see that $M_1(\gamma'^{\frac{1-n}{2}})^T M_2(\gamma')^T$ is lower triangular with diagonal entries $\gamma'^{\frac{1-n}{2}}\gamma'^k$ for $k = 0, \cdots, n-1$ which implies that  $$\lim_{\epsilon \to 0}\left[N(\gamma,\epsilon)\right] =  \gamma'^{\frac{1-n}{2}}\mbox{diag}(1,\gamma',\cdots,\gamma'^{n-1}).$$ And after setting $z = z_0$, the above equation becomes
\begin{align*}
\rho_\epsilon(\gamma)M(\epsilon)^{-1} &= Y_0(z_0,\epsilon)^{-1}N(\gamma(z_0),\epsilon)Y_0(z_N,\epsilon)M(\epsilon)^{-1}\\  &= Y_0(z_0,\epsilon)^{-1}N(\gamma(z_0),\epsilon)Y_N(z_N,\epsilon), 
\end{align*}

which together with \eqref{e1} and \eqref{e2} yields 
\begin{align*}
\lim_{\epsilon \to 0} \rho_\epsilon(\gamma)M(\epsilon)^{-1} &= \lim_{\epsilon \to 0} \left[ Y_0(z_0,\epsilon)^{-1}N(\gamma(z_0),\epsilon)Y_N(z_N,\epsilon) \right] \\  &= \lim_{\epsilon \to 0} \left[ Y_0(z_0,\epsilon)^{-1}\right] \lim_{\epsilon \to 0}\left[N(\gamma(z_0),\epsilon)\right] \lim_{\epsilon \to 0}\left[Y_N(z_N,\epsilon) \right] \\ &= \left[\omega_n(z_0)^{\frac{1-n}{2n}}P_0(z_0) \right]^{-1}  \gamma'^{\frac{1-n}{2}}\mbox{diag}(1,\gamma',\cdots,\gamma'^{n-1})\omega_n(z_N)^{\frac{1-n}{2n}}P_0(z_N) \\ &= P_0(z_0)^{-1}\mbox{diag}(1,\gamma',\cdots,\gamma'^{n-1})P_0(\gamma(z_0)).
\end{align*}

Here the last line uses that $z_N = \gamma(z_0)$ and $\omega_n(\gamma(z_0))\gamma'(z_0)^{n} = \omega_n(z_0)$. Recall that $$P_0(z) = 
\mbox{diag}(1,\omega_n^{\frac{1}{n}},\cdots,\omega_n^{\frac{(n-1)}{n}})C_\lambda$$
for some constant matrix $C_\lambda$. This implies that $$P_0(z) = \mbox{diag}(1,\gamma',\cdots,\gamma'^{n-1})P_0(\gamma(z))
$$
Finally we have $$ \lim_{\epsilon \to 0} \rho_\epsilon(\gamma)M(\epsilon)^{-1} = \mbox{I}_n$$
as desired. This concludes the proof of theorem \ref{thm1}.
\end{proof}
As a corollary, using the same notation as in the above theorem, we can calculate the growth rate of $||\rho_t(\gamma)||$. 
\begin{corollary} Let $\lambda_1, \cdots \lambda_n$ denote the $n$-th roots of $\omega_n$. Then there exists a choice $\alpha_1, \cdots,\alpha_N$ so that 
$$ \lim_{t \to \infty} \frac{\log ||\rho_t(\gamma)||}{t^{\frac{1}{n}}} = Re\left(\int_{\gamma_1}\lambda_{\alpha_1} + \cdots + \int_{\gamma_n} \lambda_{\alpha_N}\right).$$
\end{corollary}

\begin{proof}

In the proof of the theorem we showed $$\lim_{\epsilon \to 0} \rho_\epsilon(\gamma)(M(\epsilon))^{-1}= \mbox{I}_{n}, $$ where $$M(\epsilon) = e^{\frac{1}{\epsilon}\int_{q_1}^{z_0}B_0(t)dt}A_1e^{-\frac{1}{\epsilon}\int_{q_1}^{q_2}B_0(t)dt} \cdots A_Ne^{-\frac{1}{\epsilon}\int_{q_N}^{z_N}B_0(t)dt}.$$ Consider the term $$e^{-\frac{1}{\epsilon}\int_{q_i}^{q_{i+1}}B_0(t)dt},$$ recall that $B_0(t) = \mbox{diag}(\lambda_1,\cdots,\lambda_n)$. Let $\alpha_i$ be the index so that $Re(\int_{\gamma_i} \lambda_{\alpha_i}) > Re(\int_{\gamma_i} \lambda_j)$ for $j \neq \alpha_i$. Then we have the following
\begin{align*}
e^{-\frac{1}{\epsilon}\int_{q_i}^{q_{i+1}}B_0(t)dt} &= \mbox{diag}(e^{-\frac{1}{\epsilon}\int_{q_i}^{q_{i+1}} \lambda_1},\cdots,e^{-\frac{1}{\epsilon}\int_{q_i}^{q_{i+1}} \lambda_n} ) \\
 &= e^{-\frac{1}{\epsilon}\int_{q_i}^{q_{i+1}} \lambda_{\alpha_i}}\mbox{diag}(e^{-\frac{1}{\epsilon}\int_{q_i}^{q_{i+1}} \lambda_1-\lambda_{\alpha_i}},\cdots,1, \cdots,e^{-\frac{1}{\epsilon}\int_{q_i}^{q_{i+1}} \lambda_n-\lambda_{\alpha_i}} )\\
  &= e^{\frac{1}{\epsilon}\int_{q_i}^{q_{i+1}} \lambda_{\alpha_i}}(E_{\alpha_i,\alpha_i} +\mbox{diag}(e^{-\frac{1}{\epsilon}\int_{q_i}^{q_{i+1}} \lambda_1-\lambda_{\alpha_i}},\cdots,0, \cdots,e^{-\frac{1}{\epsilon}\int_{q_i}^{q_{i+1}} \lambda_n-\lambda_{\alpha_i}} ))   
\end{align*}
Here $E_{ij}$ denotes the matrix with 1 in the $ij -$entry and 0 elsewhere. The inequality $Re(\int_{q_i}^{q_{i+1}} \lambda_j - \lambda_{\alpha_i}) > 0$ implies that $$\mbox{diag}(e^{-\frac{1}{\epsilon}\int_{q_i}^{q_{i+1}} \lambda_1-\lambda_{\alpha_i}},\cdots,0, \cdots,e^{-\frac{1}{\epsilon}\int_{q_i}^{q_{i+1}} \lambda_n-\lambda_{\alpha_i}}) \rightarrow 0.$$ This gives us 
$$ e^{-\frac{1}{\epsilon}\int_{q_i}^{q_{i+1}}B_0(t)dt} = e^{\frac{1}{\epsilon}\int_{q_i}^{q_{i+1}} \lambda_{\alpha_i}}(E_{\alpha_i,\alpha_i} +O(\epsilon))$$

Consider the first three terms $e^{\frac{1}{\epsilon}\int_{q_1}^{z_0}B_0(t)dt}A_1e^{-\frac{1}{\epsilon}\int_{q_1}^{q_2}B_0(t)dt}$ in $M(\epsilon)$ $$e^{\frac{1}{\epsilon}\int_{q_1}^{z_0}B_0(t)dt}A_1e^{-\frac{1}{\epsilon}\int_{q_1}^{q_2}B_0(t)dt} = e^{\frac{1}{\epsilon}\int_{\gamma_1} \lambda_{\alpha_1}}e^{\frac{1}{\epsilon}\int_{\gamma_2} \lambda_{\alpha_2}}(E_{\alpha_1,\alpha_1} + O(\epsilon))A_1(E_{\alpha_2,\alpha_2} + O(\epsilon)).$$ There are two different cases depending on whether $\alpha_1 = \alpha_2$ or not. If $\alpha_1 = \alpha_2$ then since $A_1 = \mbox{I} + aE_{\alpha \beta}$ and $\alpha \neq \beta$ we can multiply the terms to see that $$ e^{\frac{1}{\epsilon}\int_{q_1}^{z_0}B_0(t)dt}A_1e^{-\frac{1}{\epsilon}\int_{q_1}^{q_2}B_0(t)dt} = e^{\frac{1}{\epsilon}\int_{\gamma_1} \lambda_{\alpha_1}}e^{\frac{1}{\epsilon}\int_{\gamma_2} \lambda_{\alpha_2}}(E_{\alpha_1,\alpha_1} + O(\epsilon)).$$ 
If $\alpha_1 \neq \alpha_2$ then the matrix $A_1$ must be of the form $A_1 = \mbox{I} + aE_{\alpha_1,\alpha_2}$ because $\gamma$ is crossing a Stokes lines $l_{\alpha_1, \alpha_2}$ and we find $$ e^{\frac{1}{\epsilon}\int_{q_1}^{z_0}B_0(t)dt}A_1e^{-\frac{1}{\epsilon}\int_{q_1}^{q_2}B_0(t)dt} = e^{\frac{1}{\epsilon}\int_{\gamma_1} \lambda_{\alpha_1}}e^{\frac{1}{\epsilon}\int_{\gamma_2} \lambda_{\alpha_2}}(aE_{\alpha_1,\alpha_2} + O(\epsilon)).$$ Continuing this way, considering three terms at a time, we can see that for some constant k, we may write $M(\epsilon)$ as $$M(\epsilon) = e^{\frac{1}{\epsilon}\int_{\gamma_1} \lambda_{\alpha_1}}e^{\frac{1}{\epsilon}\int_{\gamma_2} \lambda_{\alpha_2}}\cdots e^{\frac{1}{\epsilon}\int_{\gamma_N} \lambda_{\alpha_N}}(kE_{\alpha_1,\alpha_N} + O(\epsilon)).$$
Finally we can look at
\begin{align*}
\rho_t(\gamma) &= \rho_t(\gamma)M(\epsilon)^{-1}M(\epsilon) \\  &= e^{\frac{1}{\epsilon}\int_{\gamma_1} \lambda_{\alpha_1}}e^{\frac{1}{\epsilon}\int_{\gamma_2} \lambda_{\alpha_2}}\cdots e^{\frac{1}{\epsilon}\int_{\gamma_N} \lambda_{\alpha_N}}\rho_t(\gamma)M(\epsilon)^{-1}(kE_{\alpha_1,\alpha_N} + O(\epsilon)),
\end{align*}

we see that $$\frac{\log ||\rho_t(\gamma)||}{t^{\frac{1}{n}}} = Re\left(\int_{\gamma_1}\lambda_{\alpha_1} + \cdots + \int_{\gamma_N} \lambda_{\alpha_N}\right) + \frac{\log(||\rho_t(\gamma)M(\epsilon)^{-1}(kE_{\alpha_1,\alpha_n} + O(\epsilon))||)}{t^{\frac{1}{n}}}.$$ Thus since $ \log(||\rho_t(\gamma)M(\epsilon)^{-1}(kE_{\alpha_1,\alpha_N} + O(\epsilon))||) < \infty,$ by the Theorem \ref{thm1}, we find
$$ \lim_{t \to \infty} \frac{log ||\rho_t(\gamma)||}{t^{\frac{1}{n}}} = Re(\int_{\gamma_1}\lambda_{\alpha_1} + \cdots + \int_{\gamma_N} \lambda_{\alpha_N}),$$
as required.

\end{proof}

\section{Properties of the Equivariant Map}
\label{building}
In this section we will first define a $\rho_t$-equivariant map \begin{equation*}
Ep_t: \tilde{X} \setminus Z_{\omega_n} \rightarrow \mbox{SL}_n(\mathbb{C}) / \mbox{SU}(n)
\end{equation*} in terms of solutions to \eqref{eq5}, and show that this family of maps induces a map $$Ep: \tilde{X} \rightarrow \mbox{Cone}( SL_n(\mathbb{C}) / SU(n)).$$ Then we will constuct a space $E_{\omega_n}$ together with a map  $f: \tilde{X} \rightarrow E_{\omega_n}$ so that the map $Ep$ factors through $E_{\omega_n}$. Finally we will show that the map $Ep$ is locally injective for $n > 2$. 

As before consider $(q_2,q_3,\cdots, q_n) = (0,0, \cdots, t\omega_n) \in \mathcal{H}_n$ where $ \omega_n \in H^0(X,\mathcal{K}^n)$ with simple zeros and $ t \in \mathbb{R}$. In this case, equation \eqref{eq0} reduces to \begin{equation} \label{eq5} y^{(n)} + t\omega_n y = 0.
\end{equation} Let $\rho_t$ be the corresponding Oper representation described in section \ref{bg}. 

\subsection{Definition of the Equivariant Map} Given a basis $y_1(z,t),\cdots,y_n(z,t)$ to \eqref{eq5} we can define a map to SL$_n(\mathbb{C})$ by the matrix Wronskian. 

Now we denote the matrix Wronskian of a basis of solutions to \eqref{eq5} by 
$$W_t(z) = W(y_1(z,t),...,y_n(z,t)).$$

$W_t(z)$ satisfies the follow properties 
\begin{enumerate}
\item $ \frac{d}{dz} W_t(z) = W_t(z)A(z)$
\item $\det(W_t(z)) = 1$
\item $\rho_t(\gamma) W_t(z) = W_t(\gamma'^{\frac{1-n}{2}} y_1(\gamma(z),t),\cdots,\gamma'^{\frac{1-n}{2}} y_n(\gamma(z),t))$
\item $W_t(z) = C\tilde{W}(z)$ for some matrix C independent of z.
\end{enumerate}

Here  $\gamma \in \pi_1(X)$, $A(z) = \left(\begin{array}{cccc} 0 & & & -\omega_n\\ 1 & 0 & & \\ & & \ddots & \\ & & 1 & 0 \end{array} \right)$ and $\tilde{W}(z)$ is any other solution to property (1). Note that $W_t(z)^T$ satisfies \eqref{eq2}. The Wronskian matrix $W_t$ then defines a map from $\tilde{X}$ to SL$_n(\mathbb{C})$ and by projection (and with an abuse of notation) we regard it as a map to the symmetric space SL$_n(\mathbb{C}) / \mbox{SU}(n)$. 

Note that  $W_t(z)$ is \textit{not} $\rho_t$-equivariant. If $\gamma$ is a deck transformation on $\tilde{X}$, the matrix $W_t$ transforms in the following way:

\begin{align*}W_t(\gamma(z)) =& W(y_1,...,y_n)\circ (\gamma(z)) \\ =&W(y_1(\gamma(z)),...,y_n(\gamma(z)))M_2^{-1}(\gamma'(z)) \\ =& W(\gamma'^{\frac{1-n}{2}}y_1(\gamma(z)),..., \gamma'^{\frac{1-n}{2}}y_n(\gamma(z)))M_1^{-1}(\gamma'^{\frac{1-n}{2}})M_2^{-1}(\gamma'(z)) \\ =& \rho_t(\gamma)W(y_1,...,y_n)M^{-1}(\gamma)
\end{align*} where $M^{-1}(\gamma)= M_1^{-1}(\gamma'^{\frac{1-n}{2}})M_2^{-1}(\gamma'(z))$ and $M_1$, $M_2$ are the matrices defined in chapter \ref{bg}.

We wish to modify the map $W_t$ to obtain an equivariant map $Ep_t$. Thus we wish  to find a matrix $E(z)$ such that \begin{equation} \label{equivariant}
 E(z) = E(\gamma(z))M(\gamma)
\end{equation} so that $Ep_t(z) = W_t(z)E^{-1}(z)$ is actually equivariant. If we define $E(z)$ to be  $$E(z) = M_1(\omega_n^{\frac{1-n}{2n}})M_2(\omega_n^{\frac{1}{n}})$$ we have the following claim.

\begin{lemma}
E(z) satisfies \eqref{equivariant}, so that $$W_tE^{-1} = Ep_t : \tilde{X} \setminus Z_{\omega_n} \rightarrow \mbox{SL}_n(\mathbb{C}) / SU(n)$$ is well defined and equivariant. 
\end{lemma}
\begin{proof}
With $M(\gamma) = (M_{ij})$ and $E(z) = (E_{ij})$ defined above we note that they satisfy the following properties. 

\begin{enumerate} 
    \item  $M_{00}(z) = \gamma'^{\frac{1-n}{2}}, \quad E_{00} = \omega_n^{\frac{1-n}{2n}}(z)$
    \item  for $i >j, \mbox{we have}  \quad M_{ij} = 0, \quad E_{ij} = 0$

    \item for $ i \leq j, \mbox{we have}  \quad M_{ij} = M_{i,j-1}'(z)+ M_{i-1,j-1}(z)\gamma'(z)$
    \item for $ i \leq j, \mbox{we have}  \quad E_{ij} = E_{i,j-1}'(z) + E_{i-1,j-1}(z)\omega_n^{\frac{1}{n}}(z).$

\end{enumerate}

We want to show that $E(\gamma(z))M(\gamma) = E(z)$. We show it for each entry. We will prove $$E_{ij}(z) = (E(\gamma(z))M(\gamma))_{ij}$$ for $i \leq j$ by induction on $j$. For $j = 0$ we have $$ E_{00}(z) =\omega_n^{\frac{1-n}{2n}}(z) = \omega_n^{\frac{1-n}{2n}}(\gamma(z))\gamma'^{\frac{1-n}{2}} = (E(\gamma(z))M(\gamma))_{00}.$$ Now fix $j$ and suppose that for all $i \leq j$ we have $$E_{ij}(z) = (E(\gamma(z))M(\gamma))_{ij}$$ we will show that $$E(z)_{i,j+1} = (E(\gamma(z))M(\gamma))_{i,j+1}.$$ 
Using the property (4) we know that $$E_{i,j+1}(z) = E'_{i,j}(z) + E_{i-1,j}(z)\omega_n^{\frac{1}{n}}(z).$$
By induction this equals
\begin{equation*}
    = ( \sum_{k=1}^j E_{ik}(\gamma(z))M_{kj}(z))' + \sum_{k= i-1}^j E_{i-1,k}(\gamma(z))M_{kj}(z)\omega_n^{\frac{1}{n}}(z)). 
\end{equation*}
Expanding the first term using the product rule 
\begin{multline*}
    = \sum_{k =i}^j E'_{ik}(\gamma(z))\gamma'(z) M_{kj}(z) 
+ \sum_{k=i}^j E_{ik}(\gamma(z))M'_{kj}(z) \\+ \sum_{k=1-1}^j E_{i-1,k}(\gamma(z))\omega_n^{\frac{1}{n}}(\gamma(z))\gamma'(z)M_{kj}(z).
\end{multline*}

We can now combine the first and last term
\begin{multline*}
    = \sum_{k=1}^j \left[ E'_{ik}(\gamma(z)) + E_{i-1,k}(\gamma(z))(\gamma(z)) \omega_n^{\frac{1}{n}}(\gamma(z)) \right] \gamma'(z) M_{kj}(z) \\+ \sum_{k =i}^j E_{ik}(\gamma(z))M'_{kj}(z) 
+ E_{i-1,i-1}(\gamma(z))\omega_n^{\frac{1}{n}}\gamma'(z)M_{i-1,j}(z).
\end{multline*}

Using property (4) above we find
\begin{multline*}
    = \sum_{k=1}^j E_{i,k+1}(\gamma(z)) \gamma'(z)M_{kj}(z) + \sum_{k=1}^j E_{ik}(\gamma(z))M_{kj}'(z) \\ + E_{i-1,i-1}(\gamma(z))\omega_n^{\frac{1}{n}}\gamma'(z)M_{i-1,j}(z).
\end{multline*}

Applying property (3) to the first terms yields
\begin{multline*}
    = \sum_{k=i}^j E_{i,k+1}(\gamma(z))\left[ 
M_{k+1,j+1}(z) - M'_{k+1,j}(z) \right] + \sum_{k=1}^j E_{ik}(\gamma(z))M'_{kj}(z) \\ + E_{i-1,i-1}(\gamma(z))\omega_n^{\frac{1}{n}}\gamma'(z)M_{i-1,j}(z).
\end{multline*}
After canceling out terms we are left with
\begin{multline*}
    = \left(\sum_{k=1}^j E_{i,k+1}(\gamma(z))M_{k+1,j+1}(z) \right) + E_{ii}(\gamma(z))M'_ij(z) - E_{i,j+1}(\gamma(z))M'_{j+1,j}(z) \\ + E_{i-1,i-1}(\gamma(z))M_{i-1,j}(z).
\end{multline*}

Shifting the index on the first term
\begin{multline*}
    = \left(\sum_{k = i+1}^{j+1} E_{ik}(\gamma(z))M_{k,j+1}(z)  \right) + E_ii(\gamma(z))M'_ij(z) - E_{i,j+1}(\gamma(z))M'_{j+1,j}(z) \\ + E_{i-1,i-1}(\gamma(z))M_{i-1,j}(z).
\end{multline*}
Adding and subtracting the term $E_{ii}(\gamma(z))M_{i,j+1}(z)$ gives us
\begin{multline*}
    = \sum_{k =1}^{j+1} E_{ik}(\gamma(z))M_{k,j+1}(z) - E_{ii}(\gamma(z))M_{i,j+1}(z) + \\E_{ii}(\gamma(z))M'_ij(z) - E_{i,j+1}(\gamma(z))M'_{j+1,j}(z) \\ + E_{i-1,i-1}(\gamma(z))M_{i-1,j}(z)
\end{multline*}
We can see that the first term is the $(i,j+1)$ of the matrix $E(\gamma(z))M(\gamma)$

\begin{multline*}
    = (E(\gamma(z))M(\gamma))_{i,j+1} + E_{ii}(\gamma(z))M_{ij}(z) + E_{i-1,j-1}(\gamma(z))\omega_n^{\frac{1}{n}}(\gamma(z))\gamma'(z)M_{i-1,j} \\- E_{ii}(\gamma(z))M_{i,j+1}(z).
\end{multline*}
Now we can combine the two $E_{ii}$ terms
\begin{multline*}
    = (E(\gamma(z))M(\gamma))_{i,j+1} \\+ E_{ii}(\gamma(z))\left[ M'_{ij}(z) -M_{i,j+1}(z) \right] + E_{i-1,i-1}(\gamma(z))\omega_n^{\frac{1}{n}}(\gamma(z))'\gamma'(z)M_{i-1,j}(z).
\end{multline*}
Applying  property (3) above and then factoring gives us the following
\begin{equation*}
    = (E(\gamma(z))M(\gamma))_{i,j+1} +\left[ E_{i-1,i-1}(\gamma(z))\omega_n^{\frac{1}{n}}(\gamma(z)) -E_{ii}(\gamma(z)) \right ] \gamma'(z)M_{i-1,j}(z).
\end{equation*}
Now property (4) tells us that the last term is 0, so we are left with
$$ = (E(\gamma(z))M(\gamma))_{i,j+1}$$
which is what we wanted to show. 

\end{proof}

Notice that $E(z)^{-1}$ is not defined for $z \in Z_{\omega_n}$, the zero set of $\omega_n$, so that $$Ep_t: \tilde{X} \setminus Z_{\omega_n} \rightarrow \mbox{SL}_n(\mathbb{C}) / \mbox{SU}(n),$$ is only defined in the complement $\tilde{X} \setminus Z_{\omega_n}$ of the zero set $Z_{\omega_n}$. This deficiency will not obstruct the extension of the limiting map $Ep$ across the zero set $Z_{\omega_n}$.

\begin{example}
For $n = 3$, if we let $q = \omega_3$ then we find that 
$$ Ep_t(z) = W_t(z)\frac{1}{q^{\frac{1}{3}}}\left( \begin{array}{ccc} 1 & 0 & 0 \\ \frac{q'}{3} & q^{\frac{1}{3}} & 0 \\ \frac{q''}{3q} - (\frac{2q'}{3q})^2 & \frac{q'}{3} & q^{\frac{2}{3}} \end{array} \right)^{-1}$$
\end{example}

The next claim is that $Ep_t$ induces a map $Ep$ to the asymptotic cone on the entire  universal cover $\tilde{X}$.

\begin{theorem*}
The family of maps $Ep_t$ induce a map $$Ep: \tilde{X} \rightarrow \mbox{Cone}( SL_n(\mathbb{C}) / SU(n), \{ \mbox{I}_{n} \}, \{ k^{\frac{1}{n}} \} ) $$ such that for each $z \notin Z_{\omega_n}$ there exist an open set $U$ containing z so that $Ep(U) \subset A$, an apartment in Cone$(SL_n(\mathbb{C}) / SU(n))$.
\end{theorem*}
\begin{proof}
The family of maps $Ep_t$ is a one-parameter family of maps and a point $p$ in the asymptotic cone is a sequence of points $p = \{p_k \} \subset \mbox{SL}_n(\mathbb{C}) /\mbox{SU}(n)$ such that $$ \frac{d(\mbox{I}_{n},p_k)}{k^{\frac{1}{n}}} < \infty.$$ If we set $t = k$ then we may define, for $z \in \tilde{X} \setminus Z_{\omega_n},$ the map $$Ep: \tilde{X} \setminus Z_{\omega_n}  \longrightarrow \mbox{Cone}(\mbox{SL}_n(\mathbb{C}) /\mbox{SU}(n))$$ $$Ep(z) = (Ep_k(z)).$$ We need to show that for $z \in \tilde{X} \setminus Z_{\omega_n}$,  the map $ Ep(z) = (Ep_k(z)) \in \mbox{Cone}(\mbox{SL}_n(\mathbb{C}) /\mbox{SU}(n))$. That is to say, we need to show $$\frac{d(\mbox{I}_{n},Ep_k(z))}{k^{\frac{1}{n}}} < \infty.$$ To show this we will find a finite collection of points $y_i$ so taht $y_0 = z_0$, $y_N = z$ and $$\frac{d(Ep_k(y_i),Ep_k(y_{i+1}))}{k^{\frac{1}{n}}} < \infty$$ which implies that $$\frac{d(Ep_k(y_0),Ep_k(z))}{k^{\frac{1}{n}}} = \frac{d(\mbox{I}_n,Ep_k(z))}{k^{\frac{1}{n}}} < \infty.$$To begin, note that for each $y \in \tilde{X} \setminus 
Z_{\omega_n}$ by Lemma \ref{lem1} we know there exists an open set $U$ containing y and a solution $W_y(z,k)$ so that for $z \in U$ we have \begin{equation} \label{eq6}
\lim_{k \to \infty} e^{-k^{\frac{1}{n}}\int_y^z B_0(w) dw}W_y(z,k) = \omega_n^{\frac{1-n}{2n}}P_0(z). 
\end{equation} Note the the order is reversed because $W_t^T$ satisfies \eqref{eq2} and $e^{-k^{\frac{1}{n}}\int_y^z B_0(w) dw}$ is diagonal. We know that $Ep_k(z) = W_k(z)E^{-1}(z) = C(k)W_y(z,k)E^{-1}(z)$ for some matrix $C(k)$ independent of $z$ because $W_k(t)$ and $W_y(z,k)$ solve the same differential equation. Writing the map $Ep_t$ in terms of $W_y$ we find,

\begin{equation*}
\begin{split}
\frac{d(Ep_k(z),Ep_k(y))}{k^\frac{1}{n}} = & \frac{d( C(k)W_y(z,k)E^{-1}(z), C(k)W_y(y,k)E^{-1}(y))}{k^\frac{1}{n}} \\ 
= & \frac{d(W_y(z,k)E^{-1}(z),W_y(y,k)E^{-1}(y))}{k^\frac{1}{n}} 
\end{split}
\end{equation*} where the last inequality is true because $C(k)$ is an isometry on the symmetric space. Equation \eqref{eq6} shows that $W_y(y,k)$ approaches $\mbox{I}_{n}$ as $k \to \infty$ and $W_y(z,k)$ approaches $ e^{-k^{\frac{1}{n}}\int_y^z B_0(w) dw}$ as $k \to \infty$. This implies that $$d(W_y(z,k)E^{-1}(z),W_y(y,k)E^{-1}(y)) \to d(\mbox{I}_{n},e^{-k^{\frac{1}{n}}\int_y^z B_1(w) dw}) = k^\frac{1}{n}\left|\int_y^z B_1(w) dw\right|.$$
This in turns gives that us that for each $y \in \tilde{X} \setminus Z_{\omega_n}$, there is an open set $U$ so that for each $z \in U$ we have 
$$\lim_{k \to \infty} \frac{d(Ep_k(z),Ep_k(y))}{k^\frac{1}{n}} = \sqrt{2}\mbox{Re}\left(\int_y^z \omega_n^{\frac{1}{n}} \right) < \infty, \quad \mbox{for } n = 2$$
$$\lim_{k \to \infty} \frac{d(Ep_k(z),Ep_k(y))}{k^\frac{1}{n}} = \sqrt{\frac{n}{2}}\left|\int_y^z \omega_n^{\frac{1}{n}} \right| < \infty, \quad \mbox{for } n > 2$$
Now for any $y \in \tilde{X} \setminus Z_{\omega_n}$, take a path $\gamma$ from $z_0$ to $y$ and cover $\gamma$ with finitely many such sets. Pick a sequence of points $y_0 = z_0, y_1, \cdots, y_N = y$ such that $y_i, y_{i+1}$ are contained in a single open set. Then we have 

\begin{equation*}	
\lim_{k \to \infty} \frac{d(Ep_k(y_i),Ep_k(y_{i+1}))}{k^\frac{1}{n}} < \infty \Rightarrow \lim_{k \to \infty}\frac{d(Ep_k(z_0),Ep_k(y))}{k^\frac{1}{n}} = \lim_{k \to \infty}\frac{d(\mbox{I}_{n},Ep_k(y))}{k^\frac{1}{n}}< \infty 
\end{equation*}
So that $(Ep_k(y)) \in \mbox{Cone}(\mbox{SL}_n(\mathbb{C}) / \mbox{SU}(n))$. This shows that $Ep_t$ induces a map $$Ep: \tilde{X} \setminus Z_{\omega_n} \rightarrow \mbox{Cone}(\mbox{SL}_n(\mathbb{C}) / \mbox{SU}(n)).$$ Similarly we can show that the sequence $$C(k)\mbox{diag}(e^{x_1},\cdots,e^{x_n})$$ induces a map $$A \rightarrow \mbox{Cone}(\mbox{SL}_n(\mathbb{C})/\mbox{SU}(n))$$ of an apartment in the cone and the above shows that $Ep(U) \subset A$, and that the map Ep on $U$ is for $n \neq 2$, up to a constant, an isometry with respect to the $|\omega_n|^{\frac{1}{n}}$ metric. The last step to show is that the map $Ep$, which is defined on $\tilde{X} \setminus Z_{\omega_n}$, extends to $\tilde{X}$.
 
\begin{lemma} There exists a unique extension of the map $Ep$ to $\tilde{X}$. 
\end{lemma}
Let $p \in Z_{\omega_n}$ be given. Since $\omega_n$ has a simple zero by assumption, by Lemma \ref{lem3} there exists open set V containing p and a solution $W_p(z)$ so that there exists finitely many sectors $S_i$ so that for $x,y \in S_i$ 
$$d(Ep(x),Ep(y)) = \sqrt{2}\mbox{Re}\left(\int_y^z \omega_n^{\frac{1}{n}} \right) \quad \mbox{for } n = 2$$
$$d(Ep(x),Ep(y)) = \sqrt{\frac{n}{2}}\left| \int_x^y \omega_n^{\frac{1}{n}} \right| \quad \mbox{for } n > 2.$$ The last equation shows that $Ep$ restricted to $V \setminus \{ p \}$ is Lipschitz and because we are mapping to Cone$(\mbox{SL}_n(\mathbb{C}) /\mbox{SU}(n))$, which is a complete metric space, $Ep$ extends to $V$.

\end{proof}

The map $Ep$ was defined in terms of $Ep_t$ which depended on $\omega_n$, the Stokes data and solutions to \eqref{eq5}. But the limiting map $Ep: \tilde{X} \rightarrow \mbox{Cone(SL}_n(\mathbb{C})/\mbox{SU}(n))$ depends only on $\omega_n$ and the Stokes data of \eqref{eq5}. More precisely we have the following.
\begin{lemma}
There exists a space $E_{\omega_n}$, together with a map  $f: \tilde{X} \rightarrow E_{\omega_n}$ so that the map $Ep$ factors through $E_{\omega_n}$. That is to say there is a map $h: E_{\omega_n} \rightarrow \mbox{Cone(SL}_n(\mathbb{C})/\mbox{SU}(n))$ defined in terms of $\omega_n$ and Stokes data of \eqref{eq5}, so that $Ep = h \circ f$.
\end{lemma}
\begin{proof}
Given $\omega_n$ we can construct $E_{\omega_n}$ in the following way. For each $U \subset \tilde{X}$ not containing a zero and simply connected we have $n$ natural coordinates functions $\zeta_i$. For each such $U$ we have a space $E_U$ and a map $f_U: U \rightarrow E_U$ defined as $$E_U = Re(\prod_{i=1}^n\zeta_i(U))\cap A,$$ $$f_U(z) = Re(\zeta_1(z),\cdots,\zeta_n(z)), $$where $A \subset \mathbb{R}^n$ is the subspace of $\mathbb{R}^n$ with sum of coordinates equal to 0. The space $E_{\omega_n}$ is an identification space, $$E_{\omega_n} = (\coprod_{U} E_U)/ \sim,$$ where $x \sim y$  if $x \in E_U, y \in E_V$ and there exists $p_1,\cdots,p_n \in U\cap V$ so that $x = Re(\zeta_1(p_1),\cdots,\zeta_n(p_n)), y = Re(\zeta_1(p_1),\cdots,\zeta_n(p_n))$. The space $E_{\omega_n}$ is given the quotient topology. We see that $f_U$ extends to a map $f: \tilde{X} \rightarrow E_{\omega_n}$. The space $E_{\omega_n}$ is a subspace of a space constructed in \cite{katzarkov2015harmonic} and \cite{katzarkov2015constructing}.

We now want to show that there is a map $h: E_{\omega_n} \rightarrow \mbox{Cone(SL}_n(\mathbb{C})/\mbox{SU}(n))$ that depends only on $\omega_n$ and the Stokes data of \eqref{eq5}. For the base point $z_0$, we have an open set $U_0$ so that for $z \in U_0$, we have $$Ep(z) = (e^{k^\frac{1}{n}\int_{z_0}^z B_0(t)dt}).$$ We can define $h|_{E_{U_0}}$ as $h(z) = (e^{k^\frac{1}{n}\mbox{diag}(x_1,\cdots,x_n)}).$ For any $y \in \tilde{X}$ take a path from $z_0$ to $y$. By the same argument as in Theorem \ref{thm2}, there exists an open set $U$ of $z$ and a matrix $M(\epsilon)$ of the form $$ M(\epsilon) = \prod_{i=1}^{N+1} E(\gamma_{i} ,\omega_{n} )A_i,$$ so that we have matrix solution $W_t(z) = M(\epsilon)Y_N(z,t)$ so that \begin{equation*}  \lim_{\epsilon \to 0}\left(e^{-\frac{1}{\epsilon}\int_{y}^z B_0(z)}Y_N(z,\epsilon)\right) =\omega_n^{\frac{1-n}{2n}}P_0 \mbox{ with } z \in U_N. 
\end{equation*}
For $z \in U$ we see that $$Ep(z) = (M(k)e^{k^\frac{1}{n}\int_{y}^z B_0(t)dt}),$$ where the matrix $M(k)$ depends only on $\omega_n$ and the Stokes data of \eqref{eq5}.
 We can define $h|_{E_{U}}$ as $h(z) = (M(k)e^{k^\frac{1}{n}\mbox{diag}(x_1,\cdots,x_n)}).$  
 
\end{proof}
The map $Ep$ locally is given by the model equations. We can compute local properties of the map. In particular we have the following proposition.  
\begin{prop}
The map $Ep: \tilde{X} \rightarrow \mbox{Cone(SL}_n(\mathbb{C})/\mbox{SU}(n))$ is locally injective for n $>$ 2.
\end{prop} 
\begin{proof}
For each $y \in \tilde{X} \setminus 
Z_{\omega_n}$ by Lemma \ref{lem1} we know there exists an open set $U$ containing y and a solution $W_y(z,k)$ so that for $z \in U$ we have $$ \lim_{k \to \infty} e^{-k^{\frac{1}{n}}\int_y^z B_0(w) dw}W_y(z,k) = \omega_n^{\frac{1-n}{2n}}P_0(z).$$ We know that since $W_k(z)$ and $W_y(z,k)$ solve the same differential equation we have $$W_k(z) = M(k)W_y(z,k)$$ for some matrix $M(k)$. This means that for $z \in U$ we have $$Ep(z) = W_k(z)E^{-1}(z) = M(k)W_y(z,k)E^{-1}(z).$$ So that 
\begin{align*} 
\frac{d(Ep(z),M(k)e^{k^{\frac{1}{n}}\int_y^z B_0(w) dw})}{k^{\frac{1}{n}}} &=  \frac{d(M(k)W_y(z,k)E^{-1}(z),M(k)e^{k^{\frac{1}{n}}\int_y^z B_0(w) dw})}{k^{\frac{1}{n}}} \\ 
 &= \frac{d(e^{-k^{\frac{1}{n}}\int_y^z B_0(w) dw}W_y(z,k)E^{-1}(z),\mbox{I}_n)}{k^{\frac{1}{n}}} < \infty
\end{align*}
This implies that for $z \in U$, we have $Ep(z) = (M(k)e^{k^{\frac{1}{n}}\int_y^z B_0(w) dw})$. Up to an isometry we have that Ep: $ U \rightarrow \mathbb{R}^{n-1} \subset$ Cone(SL$_n(\mathbb{C})$/SU$(n)$) given by $$Ep(z) = Re\left(\begin{bmatrix} \int_y^z \omega_n^{\frac{1}{n}} dw \\ \vdots \\ \int_y^z \lambda^{(n-1)}\omega_n^{\frac{1}{n}}dw\end{bmatrix}\right).$$ This is locally injective when $z$ is not a zero of $\omega_n$ since for $x, y \in U$ we have $$d(Ep(x),Ep(y)) = \sqrt{\frac{n}{2}}\left| \int_x^y \omega_n^{\frac{1}{n}} \right|.$$ 

For $y \in Z_{\omega_n}$ a zero of $\omega_n$, by Lemma \ref{lem3}  we know there exists an open set $U$ containing y and a solution $W_y(z,k)$ so that 
\begin{equation} \label{inject}
\lim_{t \to \infty}\left(Y(z,t)e^{-t^{\frac{1}{n}}\int B_0(z)}\right) < \infty 
\end{equation} for z lying in a sector with $\zeta$-angle $< \frac{n\pi}{(n+1)}$ and has the same Stokes data as \ref{lem2}. The Stokes lines are $l_k = \{ arg({\zeta}) = \frac{k\pi}{n+1}\}$ and divide $U$ into 2($n+1$) sectors, say $S_k$. As above the Ep will be locally injective as long as the limit \ref{inject} holds. So Ep will be locally injective on the union of any $n$ adjacent sectors, $\bigcup_{k=1}^n S_{i+k}$. We have that Ep maps $S_k, \cdots, S_{n-1+k}$ to a single flat, $\mathcal{F}$. Now there is a Stokes matrix $A$ so that Ep maps sector $S_{n+k}$ to the flat, $A\mathcal{F}$, but $A\mathcal{F}$ and $\mathcal{F}$ are asymptotic along a half space so in Cone($\mbox{SL}_n(\mathbb{C})/\mbox{SU}(n))$, the map Ep maps $S_k,\cdots, S_{n+k}$ to a single flat and so Ep is injective on any $n+1$ adjacent sectors or any half plane. Since there are only 2$(n+1)$ total sectors we need to show that $S_k$ and $S_{n+k+1} = -S_k$ do not map to the same flat. Let $\zeta$ be a natural coordinate so that $\omega_n = \zeta d\zeta^n$ and  choose $z_1$ so that $\zeta(z_1) = r e^{\frac{(2k+1)\pi}{2n}} \in S_k$. If Ep maps $S_k$ and $-S_k$ to the same flat then we would have $$Ep(z) = Ep(-z).$$ Choose $z_2$ so that  $\zeta(z_2) = r e^{\frac{(2(k+n)-1)\pi}{2n}} \in S_{k+n}.$ Then since $z_1$ and $z_2$ are within $n+1$ adjacent sectors we have $$d(Ep(z_1),Ep(z_2)) = 2r$$ the distance in the $|\omega_n|^{\frac{1}{n}}$ metric. On the other hand $z_2$ and $-z_1$ are in adjacent sectors so we have $$d(Ep(-z_1),Ep(z_2)) = 2sin(\frac{\pi}{2n}).$$ So Ep(z) $\neq$ Ep(-z) which implies that Ep is locally injective.
\end{proof}
\bibliographystyle{plain}

\end{document}